\documentclass[a4paper, 12pt]{amsart}
\usepackage{amsmath}
\usepackage{amssymb, latexsym, amsthm}
\usepackage[all]{xy}

\newtheorem{Thm}{Theorem}[section]
\newtheorem{Prop}[Thm]{Proposition}
\newtheorem{Cor}[Thm]{Corollary}
\newtheorem{Lem}[Thm]{Lemma}

\newtheorem*{definition}{Definition}
\newenvironment{remark}{\it Remark. \rm}

\newcommand{\Q}[0]{\mathbb{Q}}
\newcommand{\F}[0]{\mathbb{F}}
\newcommand{\Z}[0]{\mathbb{Z}}
\newcommand{\N}[0]{\mathbb{N}}
\renewcommand{\O}[0]{\mathcal{O}}
\newcommand{\p}[0]{\mathfrak{p}}
\newcommand{\m}[0]{\mathrm{m}}
\newcommand{\Tr}{\mathrm{Tr}}
\newcommand{\Hom}[0]{\mathrm{Hom}}

\newcommand{\Gal}[0]{\mathrm{Gal}}
\newcommand{\Res}[0]{\mathrm{Res}}
\newcommand{\id}{\mathrm{id}}
\newcommand{\To}[0]{\longrightarrow}
\newcommand{\cl}{\mathrm{cl}}
\newcommand{\Li}{\mathrm{Li}}

\newcommand{\wt}[0]{\widetilde}

\newcommand{\wb}[0]{\overline}
\renewcommand{\c}[0]{\mathcal}
\renewcommand{\d}[0]{\text{\tiny{$\bullet$}}}
\begin{document}

\title[Field-of-Norms Functor and Hilbert Symbol]
{The Field-of-Norms Functor and the Hilbert Symbol for Higher Local Fields}
\author{Victor Abrashkin}
\address{Department of Mathematical Sciences, Durham University, Science Laboratories, 
South Rd, Durham DH1 3LE, United Kingdom}
\email{victor.abrashkin@durham.ac.uk}
\author{Ruth Jenni}
\email{r.c.jenni@durham.ac.uk}
\date{}

\begin{abstract} The field-of-norms functor is applied 
to deduce explicit reciprocity formulae for the Hilbert symbol   
in the mixed characteristic case from the explicit formula 
for the Witt symbol 
in characteristic $p>2$ in the context of higher local fields. 
Is is shown that a ``very special case'' of this 
construction gives Vostokov's explicit formula. 
\end{abstract}
\maketitle

%%%%%%%%%%%%%%%%%%%%%%%%%%%%%%%%%%%%%%%%%%%%%%%%%%%%%%%%%%%%%%%%%%%%%%%%%%%%%%%%%%%

\section*{Introduction}

Throughout this paper, $M$ and $N$ are fixed natural numbers, 
$p$ is an odd prime number, 
$W(k)$ is the ring of Witt 
vectors with coefficients in a finite field $k$ of characteristic $p$, 
$W(k)_{\Q_p}=W(k)\otimes_{\Z_p}\Q_p$, and $\sigma $ is the 
Frobenius automorphism of 
$W(k)$ induced by the $p$-th power map on $k$. 
In the main body of the paper we shall also use  
other notation from this Introduction without special reference.

Suppose $F$ is an $N$-dimensional local field of characteristic 0 
with the (first) residue field $F^{(1)}$  
(which is an $(N-1)$-dimensional local field) of characteristic $p$, 
$\bar F$ is a fixed algebraic closure of $F$ and 
$\Gamma _F=\Gal (\bar F/F)$. Note that, by definition, the last 
residue field $F^{(N)}$ is a finite field of 
characteristic $p$ which we shall denote by $k$. 
Fix a system of local parameters $\pi _1,\dots ,\pi _N$ in $F$. 
Let $v_F$ be the (first) valuation of $F$ such that $v_F(F^*)=\Z $. 
Then $v_F$ can be extended uniquely to $\bar F$ and we introduce 
for any $c\geqslant 0$, the ideals $\p _F^c=\{a\in\bar F\ |\ v_F(a)\geqslant c\}$. 

Let $F_{\d }$ be a strictly deeply ramified (SDR) fields
tower with parameters $(0,c)$, where 
$0< c\leqslant v_F(p)$. This means that 
$F_{\d }=\{F_n\ | n\geqslant 0\}$ is an increasing tower 
of algebraic extensions of $F_0=F$ such that for all $n\geqslant 0$, 

--- the last residue field of $F_n$ is $k$;

--- there is a system of local parameters  
$\pi _1^{(n)},\dots ,\pi _N^{(n)}$ in $F_n$ 
such that $\pi _1^{(n+1)p}\equiv \pi _1^{(n)}\,\mathrm{mod}\,\p_F^c$,\dots , 
$\pi _N^{(n+1)p}\equiv\pi _N^{(n)}\,\mathrm{mod}\,\p _F^c$.

The construction of the field-of-norms functor $X$ 
from \cite{Sch} attaches to $F_{\d }$ a  
field $X(F_{\d })=\c F$ of 
characteristic $p$. This field is  
the fraction field of the valuation ring 
$O_{\c F}=\underset{n}
\varprojlim \,O_{F_n}/\p _F^c$, where  
$O_{F_n}=\{a\in F_n\ |\ v_F(a)\geqslant 0\}$ 
are the (first) valuation rings of $F_n$ for all 
$n\geqslant 0$. Note that $\c F$ has a natural 
structure of an $N$-dimensional local field of 
characteristic $p$ with system of local parameters 
$\bar t_1=\underset{n}\varprojlim\pi _1^{(n)}$,\dots , 
$\bar t_N=\underset{n}\varprojlim\pi _N^{(n)}$ and 
last residue field $\c F^{(N)}=k$, i.e. 
$\c F$ is the field of formal Laurent series $k((\bar t_N))\dots ((\bar t_1))$.
 In addition,  
the field-of-norms functor $X$ 
provides us with a construction of a 
separable closure $\c F_{sep}$  of $\c F$ and 
identifies the Galois groups $\Gamma _{\c F}=
\Gal (\c F_{sep}/\c F)$ and 
$\Gamma _{F_{\infty }}=\Gal (\bar F/F_{\infty })$, 
where $F_{\infty }=\bigcup _{n\geqslant 0}F_n$.

We use the above system of local parameters 
$\bar t_1,\dots ,\bar t_N$ 
to construct an absolutely unramified lift 
$L(\c F)$ of $\c F$ of characteristic 0. Then $L(\c F)$ is 
an $(N+1)$-dimensional local field with system of 
local parameters $p,t_1,\dots ,t_N$;  
its first residue field $L(\c F)^{(1)}$ coincides 
with $\c F$ and for $1\leqslant i\leqslant N$, we have  
$t_i\,\mathrm{mod}\,p=\bar t_i$.

For any higher local field $L$, let $K_N(L)$ be its $N$-th 
Milnor $K$-group. In this paper we mainly use the topological 
versions  $K_N^t(L)$ of the Milnor $K$-groups, which have explicit  
systems of topological generators. Nevertheless, in the 
final statement we can return to Milnor $K$-groups 
due to the natural identification 
$K_N(L)/p^M=K_N^t(L)/p^M$, which we shall denote by $K_N(L)_M$.

The following maps play very important roles  
in the statement of the main result of this paper.

$\bullet $\ $\c N_{\c F/F}:K_N^t(\c F)\longrightarrow K_N^t(F)$.
\newline 
In Subsection \ref{S4} we prove that for so-called special 
SDR towers $F_{\d }$, there is a natural 
identification 
$K_N^t(\c F)=\underset{n}\varprojlim K_N^t(F_n)$, 
where the connecting morphisms are the norm maps 
$N_{F_{n+1}/F_n}:K_N^t(F_{n+1})\longrightarrow K_N^t(F_n)$. 
Then $\c N_{\c F/F}$ is the 
corresponding projection from $K^t_N(\c F)$ to $K^t_N(F)$. 
For arbitrary SDR towers $F_{\d }$ 
we prove the analogous ``modulo $p^M$'' statement under the assumption 
that a primitive $p^M$-th root of unity $\zeta _M\in F_{\infty }$. 
In particular, this gives the map 
$$\c N_{\c F/F}:K_N(\c F)_M\longrightarrow K_N(F)_M.$$ 

$\bullet $\ $\mathrm{Col} : 
K_N^t(\c F)\longrightarrow K_N^t(L(\c F))$. 
\newline 
This map is obtained as a section of the natural map 
from $K^t_N(L(\c F))$ to 
$K^t_N(L(\c F)^{(1)})=K_N^t(\c F)$. Its construction, in which 
the concept of topological $K$-groups is essential, 
is a direct generalisation 
of Fontaine's 1-dimensional construction from \cite{Fo}.  

$\bullet $\ $\theta ^1:
\m ^0\longrightarrow 
(1+\m ^0)^{\times }$. 
\newline 
Here $\m ^0$ consists of all series 
$\sum _{a>\bar 0}w_a\underline{t}^a$, which 
are 
convergent in $L(\c F)$, where 
the indices $a=(a_1,\dots ,a_N)\in\Z ^N$ are 
provided with the lexicographical ordering, 
all $w_a\in W(k)$, and $\underline{t}^a:=t_1^{a_1}\dots t_N^{a_N}$.
The map $\theta ^1$ is then a group homomorphism 
defined by the correspondence 
$$\sum _{a>\bar 0}w_a\underline{t}^a\mapsto 
\prod _{a>\bar 0}E(w_a,\underline{t}^a).$$
Here for any $w\in W(k)$, 
$$E(w,X)=\exp \left (wX+\dots +
\sigma ^n(w)X^{p^n}/p^n+\dots \right )\in W(k)[[X]]$$
is the Shafarevich generalisation of the Artin-Hasse exponential. 
Notice that the inverse of $\theta ^1$  
is the map given, for any $b\in 1+\m ^0$, 
by the correspondence 
$b\mapsto (1/p)\log(b^p/\sigma (b))$, where 
$\sigma $ is the 
continuous map induced by the Frobenius on $W(k)$ 
and $t_i\mapsto t_i^p$, for all 
$1\leqslant i\leqslant N$. 

$\bullet $\ $\gamma : (1+\m ^0)^{\times }
\longrightarrow \hat F^*_{\infty }$.
\newline 
Here $\hat F_{\infty }$ is the completion 
(with respect to the valuation $v_F$) of 
$F_{\infty }=\bigcup _{n\geqslant 0}F_n$ and 
the map $\gamma $ is the continuous map uniquely determined 
by  
$t_i\mapsto\underset{n\to\infty }\lim \pi _i^{(n)p^n}$, 
$1\leqslant i\leqslant N$.

We now state the main result of this paper. 

Let $F^{ab}$ be the maximal abelian extension of $F$, 
$\Gamma _F^{ab}=\Gal (F^{ab}/F)$ and 
$\widehat{K}_N^t(F):=\underset{L}
\varprojlim K_N^t(F)/N_{L/F}K_N^t(L)$, where 
$L$ runs over the set of all finite extensions of $F$ in  $F^{ab}$. 

We denote by $\Theta _F:
\Gamma _F^{ab}\longrightarrow\widehat{K}_N^t(F)$ the reciprocity 
map of local higher class field theory. For the field $\c F$, we 
introduce similarly $\c F^{ab}$, $\Gamma _{\c F}^{ab}$, 
$\widehat{K}_N^t(\c F)$ and $\Theta _{\c F}$. Then the 
compatibility of class field theories for the fields $F$ and $\c F$ 
via the field-of-norms functor means that there is 
the following commutative diagram 
\begin{equation}\label{E01}\xymatrix{\Gamma _{\c F}^{ab}
\ar[d]^{\iota _{\c F/F}}\ar[rr]^{\Theta _{\c F}}&& 
\widehat{K}^t_N(\c F)\ar[d]^{\widehat{\c N}_{\c F/F}}&&\ar[ll] K_N^t(\c F)
\ar[d]^{\c N_{\c F/F}}\\
\Gamma (F)^{ab}\ar[rr]^{\Theta _F}
&& \widehat{K}_N^t(F) &&\ar[ll]  K_N^t(F)
}
\end{equation}
Here $\iota _{\c F/F}$ is induced by 
the identification $\Gamma _{\c F}=\Gamma _{F_{\infty }}$ given 
by the field-of-norms functor, the horizontal maps on the 
right-hand side are the natural 
embeddings and the map $\widehat{\c N}_{\c F/F}$ is induced by $\c N_{\c F/F}$ 
on the corresponding completions. 
We prove the commutativity of the above diagram \eqref{E01} 
only for $\c F=X(F_{\d})$, where 
$F_{\d}$ is a so-called special SDR tower, cf. 
Subsection \ref{S4.3}. But under the additional 
assumption $\zeta _M\in F_{\infty }$ we prove the commutativity of 
the following ``modulo $p^M$'' version of \eqref{E01} 
for any SDR tower $F_{\d}$ 
(we use the same notation for all involved maps taken modulo $p^M$)
$$\xymatrix{\Gamma _{\c F}^{ab}/p^M
\ar[d]^{\iota _{\c F/F}}\ar[rr]^{\Theta _{\c F}}&& 
\widehat{K}_N(\c F)_M\ar[d]^{\widehat{\c N}_{\c F/F}}
&&\ar[ll] K_N^t(\c F)_M
\ar[d]^{\c N_{\c F/F}}\\
\Gamma _F^{ab}/p^M\ar[rr]^{\Theta _F}
&& \widehat{K}_N(F)_M &&\ar[ll]  K_N(F)_M
}$$

This property allows us to consider the $M$-th Hilbert pairing 
$$(\ ,\ )_M^{F_{\d }}: \hat F^*\times 
\c N_{\c F/F}(K_N^t(\c F))\longrightarrow \langle\zeta _M\rangle $$ 
under the condition that $\zeta _M\in F_{\infty }$. 
Namely, if $b\in \c N_{\c F/F}(K_N(\c F)_M)$ then 
there is $\tau\in\Gamma _{F}^{ab}/p^M$ such that 
$\tau |_{F_{\infty }}=\id $ and $\Theta _{F}(\tau )=b $. Then 
for any $a\in\hat F^*_{\infty }$, $(a,b)^{F_{\d }}_M:=\tau (\xi )/\xi $, 
where $\xi\in \bar F$ is such that $\xi ^{p^M}=a$.

Suppose $F_{\d }$ is an SDR tower with parameters $(0,c)$. 

\begin{definition} The tower $F_{\d }$ is called $\omega $-admissible,  
for $\omega\in\Z _{\geqslant 0}$, if 
$cp^{\omega }>2v_F(p)/(p-1)$ and if 
$F_{\omega }$ contains a primitive 
$p^{M+\omega }$-th root of unity 
$\zeta _{M+\omega }$. 
\end{definition}

For an $\omega $-admissible SDR tower $F_{\d }$, 
we define (not uniquely) an element 
$H_{\omega }=H_{\omega }(\zeta _{M+\omega })
\in \m ^0$ as follows. 
Suppose $H'=1+\sum _{a>\bar 0}
w_a\underline{t}^a\in 1+\m ^0$ is 
such that $\gamma (H')
\equiv\zeta _{M+\omega }\,\mathrm{mod}\,\p _F^c$. 
Then we set $H_{\omega }:=H^{\prime p^{M+\omega }}-1$. Note 
that the construction 
of $H_{\omega }$ does not require the knowledge of the whole 
tower $F_{\d }$, but only of the field $F_{\omega}$. 
In particular, if $\zeta _M\in F=F_0$ then the corresponding element 
$H_0\in\m ^0$ will be used later in the definition of Vostokov's pairing. 

With the above notaion we have,  
for any $\omega $-admissible SDR tower, the following 
explicit formula for the $M$-th Hilbert symbol.

\begin{Thm} \label{T0.2}If $f\in\m ^0$, 
$\beta\in K_N(\c F)$ and $\theta :=\gamma\circ \theta _1 $ then 
 \begin{equation}\label{E0.1a}(\,\theta (f)\,,\,\c N_{\c F/F}(\beta )\,)_M^{F_{\d}}=
\zeta ^{p^{\omega }A}_{M+\omega }
\end{equation}
where $A=(\Tr \circ \Res)\left ((f/H_{\omega })
d_{\log}\mathrm{Col}(\beta )\right ).$
\end{Thm}

Here (and everywhere below) 
$\Tr $ is the trace map for the field extension $W(k)_{\Q _p}/\Q _p$ and 
$\Res $ is $N$-dimensional residue. 

The above Theorem \ref{T0.2} gives one of most general approaches to 
the explicit formulas for the Hilbert symbol. The 
proof uses the strategy from \cite{Ab1} and 
the construction of the field-of-norms functor for  
higher local fields from \cite{Sch}. As a result, the explicit formula 
\eqref{E0.1a} 
is obtained from the explicit formula for the Witt symbol 
in characteristic $p$. Notice that symbol \eqref{E0.1a} depends not only on 
a fixed system of local parameters $\pi _1,\dots ,\pi _N$ of $F$ but also 
involves special lifts of elements  of $F$ to $L(\c F)$. 

The result of the above Theorem \ref{T0.2} is related very closely 
to Vostokov's explicit formula for the $M$-th Hilbert symbol 
$F^*\times K_N^t(F)\longrightarrow \langle\zeta _M\rangle $. 
In this formula the elements of $F^*$ appear as the results of 
the substitution $t_i\mapsto\pi _i$, $i=1,\dots ,N$, into formal Laurent 
series with coefficients in $W(k)$ and indeterminants $t_1,\dots ,t_N$. 
Vostokov's proof of this formula is based on a hard computation showing that the  
formula gives the same result for arbitrary choices of local parameters 
$\pi _1,\dots ,\pi _N$. 

In Section 3 we develop a slightly 
different approach to Vostokov's result. First of all, the Vostokov pairing 
has two different aspects. One is purely $K$-theoretic: it gives a (non-degenerate) 
pairing between $K_1(F)/p^M$ and $K_N(F)/p^M$ and factors through the canonical 
morphism 
\begin{equation}\label{E0.2a}
K_1(F)/p^M\times K_N(F)/p^M\longrightarrow K_{N+1}(F)/p^M.
\end{equation} 
(Note that Vostokov's formula gives also a pairing between 
$K_i(F)/p^M$ and $K_{N-i}(F)/p^M$  for $1<i<N$.) We establish these properties 
following the strategy from \cite{Ab1} and using an idea of one 
calculation from \cite{BV}. Note that we can work throughout with  
our fixed system of local parameters $\pi _1,\dots ,\pi _N$. 
Then the  Galois-theoretic aspect of Vostokov's pairing,  
i.e. that it coincides with the Hilbert symbol, follows by  
an easy calculation from the following two elementary facts:

--- the Hilbert symbol also factors through the map \eqref{E0.2a};

--- $K_{N+1}(F)/p^M$ is generated by one element which can be written in 
terms of our fixed system of local parameters $\pi _1,\dots ,\pi _N$. 

At the end of Section \ref{S5} we show that symbol \eqref{E0.1a} from Theorem 
\ref{T0.2} coincides with Vostokov's pairing if we use a \lq\lq very special \rq\rq\ 
SDR tower  $F^0_{\d }=\{F_n^0\ |\ n\geqslant 0\}$ such 
that $F^0_0=F$ and for all $n\geqslant 0$, $F^0_n$ has a system of local parameters 
$\pi _1^{(n)},\dots ,\pi _N^{(n)}$ with  
$\pi _i^{(n+1)p}=\pi _i^{(n)}$ 
and $\pi _i^{(0)}=\pi _i$ for all $1\leqslant i\leqslant N$. 

Note that other interpretations of Vostokov's formula have been given by K.Kato  
\cite{Ka} in terms of Fontaine-Messing theory and by S.\,Zerbes 
\cite{Ze} in terms of $(\varphi ,\Gamma )$-modules under an additional restriction on 
the basic field $F$. 
Note also the paper \cite{Fu} where special cases of 
the constructions of the field-of-norms functor in the context 
of higher local fields were treated.

The structure of the paper is as follows. 
In Section \ref{S1} we discuss basic matters: 
the concept of higher local field, the $P$-topology,  
special systems of topological generators 
for the Milnor $K$-groups and the norm 
map in the context of $K$-groups. In Section \ref{S2} we give an 
invariant approach to the concept of residue, the Witt symbol 
and the Coleman map 
in the context of higher local fields.  
In Section \ref{S3} we recover the construction of Vostokov's pairing 
following mainly the strategy of the paper \cite{Ab1}. In Section 
\ref{S4} we use the field-of-norms functor $X$ to relate the behaviour 
of topological Milnor $K$-groups in SDR towers. Finally, in Section \ref{S5} 
we prove the compatibility of the field-of-norms functor with class field theories 
for the fields $\c F=X(F_{\d })$ and $F=F_0$ and use the compatibility 
of the Kummer theory for $F$ and the Witt-Artin-Schreier theory for $\c F$ from 
\cite{Ab2} to deduce the statement of Theorem \ref{T0.2}.

\section{Preliminaries} \label{S1}

Most of the notation introduced in 
this Section will be used in the next sections without special references. 
In particular, this holds for the notation $F$, 
$\pi _1,\dots ,\pi _N$, $\c F$, $\bar t_1,\dots ,\bar t_N$, 
$O(\c F)$ and $L(\c F)$.

\subsection{Higher local fields} \label{S1.1}

Let $L$ be an $N$-dimensional local field. This means that 
$L$ is a complete discrete valuation field and its (first) residue field 
$L^{(1)}$ is an $(N-1)$-dimensional local field. In our setting, $0$-dimensional 
local fields are finite fields of charactersitic $p$. Let $L^{(N)}$ be 
the $N$-th residue field of $L$. By inductive definition this 
means that  $L^{(N)}=(L^{(1)})^{(N-1)}$ and, therefore, 
it is a finite field of characteristic $p$. 
The system $u_1,\dots ,u_N$ is a system of local parameters of $L$, 
if $u_1$ is a local parameter of $L$, $u_2,\dots ,u_N$ 
belong to the valuation ring $O_L$ of $L$ and 
the images of $u_ 2,\dots ,u_N$ in $L^{(1)}$ 
form a system of local parameters of $L^{(1)}$. 
The field $L$ is equipped with a special topology (we call it the $P$-topology) 
which relates all $N$ valuation topologies of $L$, $L^{(1)}$, 
$L^{(2)}:=(L^{(1)})^{(1)},\dots ,L^{(N)}:=(L^{(N-1)})^{(1)}$.  
The idea how to construct such topology appeared 
first in \cite{Pa2} and then was considerably developed and studied 
in \cite{Fe, Zh1, MZh}. We can sketch its definition  
as follows. 

Fix a system of local aparameters $u_1,\dots ,u_N$ in $L$. 
Note that  any 
element $x \in L$ can be written uniquely as a formal series 
\begin{equation} \label{E1.1}
x=\sum _{a=(a_1,\dots,a_N)}\alpha_{a}u_1^{a_1}\cdots u_N^{a_N},
\end{equation} 
where all coefficients $\alpha_{a}$ are the Teichm\"uller representatives of 
the elements of $L^{(N)}$ in $L$. Here $a\in\Z ^N$ and 
there are (depending on the element $x$) 
integers $I_1, I_2(a_1),\dots,I_N(a_1,\dots,a_{N-1})$ such that 
$\alpha_{ a}=0$ if either $a_1<I_1$ or $a_2<I_2(a_1),\dots, $ 
or $a_N<I_N(a_1,\dots,a_{N-1})$. Then the $P$-topological structure 
on $L$ can be defined by induction on $N$ as follows. 
If $N=0$ then it is discrete. If $N\geqslant 1$ then 
$\bar u_2=u_2\,\mathrm{mod}\,u_1$, \dots , 
$\bar u_N=\,\mathrm{mod}\,u_N$ is a system of local parameters in $L^{(1)}$ 
and we can define a section $s:L^{(1)}\longrightarrow L$ 
by $\sum _{a}\alpha _a\bar u_2^{a_2}\dots \bar u_N^{a_N}\mapsto\sum _{a}
u_2^{a_2}\dots u_N^{a_N}$. By definition, 
the basis of open neighbourhoods $\c C_{L,\{u_1,\dots ,u_N\}}$ 
in $L$ consists of the sets $\sum _{b\in\Z }t^bs(U_b)$, where 
all $U_b\in\c C_{L^{(1)},\{\bar u_2,\dots ,\bar u_N\}}$ and 
$U_b=L^{(1)}$ if $b\gg 0$. One can prove then that this 
 does not depend on the initial choice of local parameters $u_1,\dots ,u_N$. 
Then any compact subset in $L$ is a closed subset in the compact subset 
of the form $\sum _{b\in\Z }t^bs(C_b)$, where all $C_b\subset L^{(1)}$ are compact and 
$C_b=0$ for $b\ll 0$. In particular, the set of all $\xi\in L$ given 
by $\eqref{E1.1}$ with fixed $I_1, I_2(a_1),\dots ,I_N(a_1,\dots ,a_{N-1}$, 
is compact. The following property explains that the concept of 
convergency in the $P$-topology just coincides with the concept of 
convergency of formal power series. 

\begin{Prop} A sequence $\xi _n\in L$ converges to 
 $\xi =\sum _{a}\alpha _au_1^{a_1}\dots u_N^{a_N}\in L$ if and only if 
for any $a\in\Z^N$, the sequence $\alpha _{an}$ converges to $\alpha _a$ in $k$. 
\end{Prop}

\begin{proof} The proof can be easily reduced to the case $\xi =0$. 
 
Suppose for $b\in\Z $, $\bar\xi _{bn}\in L^{(1)}$ are such that 
$\xi _n=\sum _{b}u_1^{b}s(\bar\xi _{bn})$. Clearly, 
$\underset{n\to\infty }\lim \xi _n= 0$ implies that for any $b\in\Z $, 
$\underset{n\to\infty }\lim\bar\xi _{bn}=0$. 
Therefore, by induction on $N$ we obtain that 
all $\underset{n\to\infty }\lim\alpha _{an}=0$.  

Inversely, suppose that for all $a\in\Z ^N$, $\underset{n\to\infty }
\lim\alpha _{an}=0$. Then by induction on $N$, for any 
$b\in\Z $, $\underset{n\to\infty }\lim\bar\xi _{bn}=0$. One can prove 
(again by induction on $N$) the existence of a compact $C\subset L$ 
such that all $\xi _n\in C$. (I.\,Fesenko pointed out to the first author that 
this implies that the $P$-topology is compactly generated.) In particular, 
there is $b_0\in\Z $ such that $\bar\xi _{bn}=0$ if $b<b_0$. Now take any 
$U=\sum _{b}u_1^bs(U_b)\in\c C_{L,\{u_1,\dots ,u_N\}}$. Then there is 
$b_1\in\Z $ such that $U_b=L^{(1)}$ for all $b> b_1$. For 
$b_0<b\leqslant b_1$, let $m(b)\in\Z $ be such that 
$\bar\xi _{bn}\in U_b$ if $n\geqslant m(b)$. Then for 
$n\geqslant\max\{\,m(b)\,|\,b_0<b\leqslant b_1\}$, $\xi _n\in U$, 
i.e. $\underset{n\to\infty }\lim\xi _n=0$. 
\end{proof}

In terms of the power series \eqref{E1.1}, the $N$-dimensional 
valuation ring $\c O _L$, resp. the maximal ideal $\m _L$, of $L$ 
consists of the elements $x$ such that  
all $\alpha _a=0$ if $a<\bar 0=(0,\dots ,0)$, 
resp. $a\leqslant\bar 0$, with respect to the 
lexicographic ordering. Note that $L$, $\c O_L$ and $\m _L$ are $P$-topological 
additive groups.  
Multiplication does not make $L^*$ into a topological group, but all 
operations in the field $L$ are sequentially $P$-continuous. The choice of 
local parameters $u_1,\dots ,u_N$ provides an isomorphism 
$L^*\simeq k^*\times \langle u_1\rangle\times\dots \langle 
u_N\rangle\times (1+\m _L)^{\times }$, 
where only the last factor has a non-trivial $P$-topological structure. 

The concept of $P$-topology plays a very important 
role in this paper and we refer usually to 
the papers \cite{Zh1} and \cite{MZh} 
for its detailed exposition. 
In particular, these papers contain the study of infinite products 
in $L$. 
The following fact clarifies the meaning of infinite products and 
will be used below without 
special references. Suppose $I_1,\dots ,I_N(a_1,\dots ,a_{N-1})$ 
are the above defined parameters. Consider the infinite product 
of the form 
$\underset{a>\bar 0}\prod (1+\alpha _au_1^{a_1}\dots u_N^{a_N})$, 
where, as earlier, $a=(a_1,\dots ,a_N)\in\Z ^N$, $\alpha _a$ are 
the Teichm\"uller representatives of elements of $k$ and $\alpha _a=0$ 
if either $a_1<I_1$, or $a_2<I_2(a_2)$, \dots , or $a_N<I_N(a_1,\dots ,a_{N-1})$. 
Then any such infinite product converges in 
$L^*$ and any element from $1+\m _L$ can be presented uniquely  
as a such infinite product 
(with suitably chosen parameters $I_1,\dots ,I_N(a_1,\dots ,a_N)$).

The main object we shall deal with 
is an $N$-dimensional local field $F$ of characteristic 0 with first residue 
field $F^{(1)}$ of charactersitic $p$, $N$-th residue field $k$ 
(which is necessarily finite) and  
a fixed system 
of local parameters $\pi _1,\dots ,\pi _N$. 
Fix an algebraic closure $\bar F$ of $F$, set 
$\Gamma_F=\Gal(\bar F/F)$ and denote by $\Gamma _F^{ab}$ 
the maximal abelian quotient of $\Gamma _F$. 

We also consider $N$-dimensional local fields of 
characteristic $p$ with last residue field $k$. Any such field $\c F$ is isomorphic to the field of formal Laurent power series 
$k((\bar t_N))\dots ((\bar t_1))$, where 
$\bar t_1,\dots ,\bar t_N$ is any system of local parameters of $\c F$. 
We use this system of local parameters as a $p$-basis for $\c F$ to construct a 
flat $\Z_p$-lift $O(\c F)$ of $\c F$ to characteristic 0. By definition 
$O(\c F)=\underset{n}\varprojlim\,O_n(\c F)$, where for all $n\in\N $, 
$$O_n(\c F)=W_n((t_N))\dots ((t_1))\subset W_n(\c F)$$
are $\Z /p^n$-flat lifts of $\c F$ and for $1\leqslant i\leqslant N$, 
$t_i=[\bar t_i]$ are the Teichm\"uller representatives of $\bar t_i$.

The lift 
$O(\c F)$ is a complete discrete valuation ring of  the $(N+1)$-dimensional 
local field $L(\c F)=\mathrm{Frac}\,O(\c F)$. Note that $L(\c F)^{(1)}=\c F$ and 
$L(\c F)$ has a fixed system of local parameters $p,t_1,\dots ,t_N$ 
such that for $1\leqslant i\leqslant N$, $t_i\,\mathrm{mod}\,p=\bar t_i$. 
The elements of $L(\c F)$ can be written as 
formal power series $\sum_a \gamma_a t_1^{a_1}\cdots t_N^{a_N}$ 
with natural  conditions on the coefficients $\gamma_a \in W(k)$, 
where $a=(a_1,\dots,a_N)\in \Z^N$.

\subsection{$P$-topological bases of $\c F^*$ and $F^*/F^{*p^M}$}\label{S1.2}

The concept of $P$-topology allows us to describe 
explicitly the structure of the multiplicative groups 
$\c F^*$ and of $F^*/F^{*p^M}$ under the additional assumption 
that $\zeta _M\in F$.

Consider the case of the field $\c F=k((t_N))\dots ((t_1))$. 
Choose an $\F _p$-basis $\theta _1,\dots ,\theta _s$ 
of $k\simeq\F _{p^s}$. 
Then any element of $\c F^*$ 
can be uniquely written as an infinite product as follows 
$$\gamma \bar t_1^{a_1}\dots \bar t_N^{a_N}
\prod _{j,b}(1+\theta _j\underline{t}^b)^{A_{jb}},$$ 
where $\gamma \in k^*$, $1\leqslant j\leqslant s$, 
$a_1,\dots ,a_N\in\Z $, 
$b$ runs over the set of all multi-indices 
$(b_1,\dots ,b_N)\in\Z ^N\setminus p\Z ^N$, $b>\bar 0$, 
$\underline{\bar t}^b:=\bar t_1^{b_1}\dots \bar t_N^{b_N}$, 
and all $A_{jb}\in\Z _p$. The only essential condition on 
the above infinite product is that it must converge in 
$\c F$ with respect to the $P$-topology. In particular, 
with the above notation 
the elements $\eta _{jb}:=1+\theta _j\underline{\bar t}^b$ 
form a set of free topological 
generators of 
the subgroup $(1+\m _{\c F})^{\times }$ of $\c F^*$.

Consider the case of the field $F$. In this case we have a similar 
description of the group $F^*/F^{*p^M}$ under the assumption that 
$F$ contains a primitive $p^M$-th root of unity $\zeta _M$. 

Suppose $p=\pi _1^{e_1}\dots\pi _N^{e_N}\eta =\underline{\pi }^{e}\eta $, 
where $e=(e_1,\dots ,e_N)\in\Z ^N$ and 
$\eta\in  \c O _F^*$. 
Then Hensel's Lemma implies that 
any element of $F^*$ modulo $F^{*p^M}$ appears in the form 
$$[\gamma ]\pi _1^{a_1}\dots \pi _N^{a_N}
\epsilon _0^{A_0}
\prod _{j,b}\eta _{jb}^{A_{jb}},$$ 
where 

---  $a_1,\dots ,a_N,A_0$ and all $A_{jb}$ are integers uniquely determined  
modulo $p^M$; 

--- $\eta _{jb}:=1+[\theta _j]\underline{\pi }^b$, where 
the multi-index $b=(b_1,\dots ,b_N)$ runs over the set 
of all $b\in\Z ^N\setminus p\Z ^N$ such that 
$\bar 0<b<e^*:=ep/(p-1)$;

---  $\epsilon _0=1+[\theta _0]\underline{\pi }^{e^*}$, 
where $\theta _0\in k$ is such that 
$1+[\theta _0]\underline{\pi }^{e^*}\notin (1+\m _F)^p$. 

{\it Remarks} 1) There is a more natural construction of 
the generator $\epsilon _0$ related to the concept of 
$p^M$-primary element. 
By definition, 
$\epsilon \in F^*$ is $p^M$-primary if the extension 
$F(\epsilon^{1/p^M})/F$ is purely unramified of degree $p^M$, i.e. 
the $N$-th residue fields satisfy  
$[F(\epsilon^{1/p^M})^{(N)}:F^{(N)}]=p^M$. 
Note that the images of $p^M$-primary elements in 
$F^*/F^{*p^M}$ form a cyclic group of order $p^M$.
One of first explicit constructions of $p^M$-primary 
elements was given by Hasse, cf. \cite{Vo1}, and can be explained as follows. 
Let $\xi\in\m _F$ be such that $E(1,\xi )=\zeta _M$. 
Let $\alpha _0\in W(k)$ be such that 
$\Tr (\alpha _0)=1$ and let $\beta\in W(\bar k)$ 
be such that $\sigma (\beta )-\beta =\alpha _0$. Then 
$\epsilon _0=E(\beta ,\xi )^{p^M}$ is 
a $p^M$-primary element of $F$. 
In Section \ref{S5} we shall use the $p^M$-primary element in the form 
$\epsilon _0=\theta (\alpha _0H_0)$, where 
$H_0=H_0(\zeta _M)\in \m ^0$ was defined in the Introduction. 
A natural explanation of this construction of $p^M$-primary element appears there 
as a special case of the relation between 
the Witt-Artin-Schreier and Kummer theories.

2) The original construction of the Shafarevich basis \cite{Sh} 
systematically uses the Shafarevich exponential 
$E(w,X)$ and establishes an explicit isomorphism 
$F^*/F^{*p^M}\simeq \langle \bar\pi _1\rangle ^{\Z /p^M}\times 
\langle \bar\epsilon _0\rangle ^{\Z/p^M }\times \prod _{b}
 W_M(k)_b$, where $0<b<e^*$, 
$\bar\pi _1=\pi _1\,\mathrm{mod}\, F^{*p^M}$, 
$\bar\epsilon _0=\epsilon _0\,\mathrm{mod}\, F^{*p^M}$ and $gcd(b,p)=1$, 
in the 1-dimensional case. This construction can 
be generalised to the $N$-dimensional case.
\medskip 

\medskip

\subsection{Topological Milnor $K$-groups}\label{S1.3}

For a higher local field $L$ and a positive integer 
$n$, let $K_n(L)$ be the $n$-th Milnor $K$-group of $L$. 
Let $VK_n(L)$ be the subgroup of 
$K_n(L)$ generated by the symbols having at least 
one entry in $V_L :=1+\m _L$. If 
$L$ is of dimension $N$ and 
$u _1,\dots,u _N$ is a system of 
local parameters of $L$, then, by \cite{Zh1},  
\begin{equation}\label{E1.3}
K_N(L)\simeq VK_N(L) \oplus \Z  \oplus 
\prod _{1\leqslant i\leqslant N}A_{iN}(L),
\end{equation}
where $\Z $ corresponds to the subgroup 
generated by $\{u_1,\dots,u _N\}$ and 
for all $1\leqslant i\leqslant N$, the group 
$A_{iN}(L)\simeq L^{(N)*}$ consists of the symbols  
$\{[\alpha ],u _1,\dots,u_{i-1},u _{i+1},\dots,u _N\}$ with 
$\alpha \in L^{(N)*}$. 

Following \cite{Fe, Zh1}  we 
introduce the $P$-topology on $K_N(L)$ as follows. 
The topology on $VK_N(L)$ is defined to be 
the finest topology such that the map 
of topological spaces $V_L \times (L^*)^{N-1}\to VK_N(L)$ 
is sequentially continuous. 
The other direct summands in \eqref{E1.3} are 
equipped with the discrete topology. Then the 
topological Milnor $K$-groups $K_N^t(L)$ are 
defined to be $K_N(L)/\Lambda$, where $\Lambda$ 
is the intersection of all neighbourhoods 
of zero, with the induced topology. 
By \cite{Fe}, 
$\Lambda=\bigcap_{n\ge 1}nVK_N(L)=
\bigcap_{m\ge 1}p^mVK_N(L)$, using $l$-divisibility 
of $VK_N(F)$, for any $l$ prime to $p$. In particular, for any $M\geqslant 1$, 
$K_N^t(L)/p^M=K_N(L)/p^M$ and the decomposition \eqref{E1.3} 
induces the decomposition \linebreak 
$K_N^t(L)\simeq\Z\oplus VK_N^t(L)$. 

The advantage of the topological $K$-groups 
$K_N^t(L)$ is that they admit $P$-topological generators 
analogous to those of the multiplicative group 
$L^*$ from Subsection \ref{S1.2}.
Before stating these results notice that 
for any higher local field $K$ one can introduce 
a filtration of $K_N^t(K)$ by the subgroups 
$U_K^cK_N^t(K)$, where $c\geqslant 0$. These subgroups 
are generated by the symbols $\{\alpha _1,
\dots ,\alpha _N\}\in K_N^t(K)$ such that 
$v_K(\alpha _1-1)\geqslant c$. Here $v_K$ 
is the 1-dimensional valuation on $K$ such that $v_K(K^*)=\Z$.  
Then the 
classical identity 
\begin{equation}\label{E1.3'}
 \{1-\alpha ,1-\beta\}=\{\alpha (1-\beta ),1+\alpha\beta (1-\alpha )^{-1}\}
\end{equation}
for 2-dimensional Milnor $K$-groups implies that 
$$\{\alpha _1,\dots ,\alpha _N\}\in U_K^{c_1+\dots +c_N}K_N^t(K)$$  
if  
$v_K(\alpha _i-1)\geqslant c_i$ 
for $1\leqslant i\leqslant N$. 
\medskip

$\bullet $\ {\it Generators of $K_N^t(\c F)$.}

 For 
$a=(a_1,\dots ,a_N)\in\Z ^N$, $a\notin p\Z ^N$, 
$a>\bar 0$, let $1\leqslant i(a)\leqslant N$ be 
such that $a_1\equiv \dots \equiv a_{i(a)-1}\equiv 0\,\mathrm{mod}\,p$ 
but $a_{i(a)}\not\equiv 0\,\mathrm{mod}\,p$. 
As earlier, choose an $\F_p$-basis $\theta _1,\dots ,\theta _s$ of $k$ 
and for all above multi-indices $a$ and 
$1\leqslant j\leqslant s$, set 
\begin{equation}\label{E1.3a}
\varepsilon _{ja}=\{1+\theta _j\underline{\bar t}^a,
\bar t_1,\dots,\bar t_{i(a)-1},\bar t_{i(a)+1},
\dots,\bar t_N\}.
\end{equation}
This is a system of free topological generators of $VK_N^t(\c F)$ and 
$K_N^t(\c F)=VK_N^t(\c F)\oplus\langle\varepsilon _0\rangle $, 
where $\varepsilon _0=\{\bar t_1,\dots ,\bar t_N\}$. 
This means that any element 
 $\xi\in K_N^t(\c F)$ can be written in the form 
$\xi =A_0\varepsilon _0+\sum _{j,b}A_{jb}\varepsilon _{jb}$,  
where $A_0$ and all $A_{jb}$ belong to $\Z _p$ and,  
for any $1\leqslant i_0\leqslant N$, the infinite product 
$$\underset{j,b, i_0(b)=i_0}
\prod (1+\theta _j\underline{t}^b)^{A_{jb}}$$ 
converges in $\c F$. 
This can 
be obtained from relation \eqref{E1.3'}. Moreover, for a given $\xi\in K_N^t(\c F)$,  
the corresponding coefficients $A_0$ and 
$A_{jb}$ are uniquely determined by $\xi $, in other words the above system of symbols 
$\varepsilon _0$ and $\varepsilon _{jb}$ is a system of free topological 
generators for $K_N^t(\c F)$. 
This was established by Parshin \cite{Pa2} via 
an analogue of the Witt pairing, cf. Subsection \ref{S2.2} below. 
It can be also deduced from 
the Bloch-Kato theorem \cite{BK}, which gives an 
explicit description of the grading of the filtration 
$U_{\c F}^c(K_N^t(\c F))$, $c\geqslant 0$. 

\medskip 

$\bullet $\ {\it Generators of $K_N^t(F)/p^M$, $\zeta _M\in F$.}

Introduce similarly 
the elements 
\begin{equation}\label{E1.3c}
\varepsilon _{ja}=\{1+[\theta _j]
\underline{\pi }^a,\pi _1,\dots,\pi _{i(a)-1},\pi _{i(a)+1},
\dots,\pi _N\},
\end{equation} 
where $1\leqslant j\leqslant s$, $a\in\Z ^N\setminus p\Z ^N$ and 
$\bar 0<a<e^*:=ep/(p-1)$. Set $\varepsilon _0=\{\pi _1,\dots ,\pi _N\}$  
and for $1\leqslant i\leqslant N$, 
\begin{equation}\label{E1.3b}
\varepsilon _{ie^*}=\{\epsilon _0,\pi _1,\dots ,
\pi _{i-1},\pi _{i+1},\dots ,\pi _N\},
\end{equation}
where $\epsilon _0$ was defined in Subsection \ref{S1.2}. 

Then for similar reasons to the case $L=\c F$, the above elements $\varepsilon _0$, 
$\varepsilon _{ja}$ and $\varepsilon _{ie^*}$ 
give a set of $P$-topological generators 
of the $\Z /p^M$-module $K_N^t(F)/p^M$. 
The Bloch-Kato theorem 
\cite{BK} about the gradings of $U^c_F(K_N^t(F))$, where 
$0\leqslant c\leqslant e'=v_F(p)p/(p-1)$, implies 
that the system of topological generators 
\eqref{E2.2a} is a topological $\Z /p$-basis of $K_N^t(F)/p$. 
The fact that we have a system of $\Z /p^M$-free 
topological generators can be deduced from the description of 
$p$-torsion in $K_N(F)$ from \cite{Fe}. This fact 
can be also established directly from the non-degeneracy of  
Vostokov's pairing, cf. Section \ref{S3}. 
\medskip 

$\bullet $\ {\it Generators of $K_{N-1}^t(F)/p^M$ and 
$K_{N+1}^t(F)/p^M$, $\zeta _M\in F$.}

A similar technique can be used to prove that 
$K_{N-1}^t(F)/p^M$ is topologically generated by the elements of the form: 
\medskip 

--- $\{1+[\theta _j]\underline{\pi }^a,\pi _{j_1},\dots ,\pi _{j_{N-2}}\}$, 
where $1\leqslant j\leqslant s$, $a\in\Z ^N\setminus p\Z ^N$, $\bar 0<a<e^*$, 
$1\leqslant j_1<\dots <j_{N-2}\leqslant N$ and 
$i(a)\notin \{j_1,\dots ,j_{N-2}\}$. 
\medskip 

--- $\{\epsilon _0,\pi _{j_1},\dots ,j_{N_2}\}$, where 
$1\leqslant j_1<\dots <j_{N-2}\leqslant N$. 
\medskip 

Similarly, in the case of $K_{N+1}/p^M$ 
we have only one generator given by the symbol 
$\{\epsilon _0,\pi _1,\dots ,\pi _N\}$. 
\medskip

\subsection{The Norm map}\label{S1.4}

For a finite extension of higher local fields $L/K$, 
the Norm-map of Milnor $K$-groups $N_{L/K}: K_n(L)
\to K_n(K)$ was defined in \cite{BT} and \cite{FV}. 
It has the following properties:

\begin{enumerate}
\item if $\alpha _1\in L^*$ and $\alpha _2,\dots ,\alpha _n\in K^*$ then 
$$N_{L/K}\{\alpha _1,\alpha _2,\dots ,\alpha _n\}=
\{N_{L/K}(\alpha _1),\alpha _2,\dots ,\alpha _n\};$$

\item for a tower of finite field extensions 
$F \subset M \subset L$, it holds $N_{L/F}=N_{L/M}\circ N_{M/F}$; 

\item if $i_{L/K}:K_n(K)\longrightarrow K_n(L)$ is 
induced by the embedding $K\subset L$ then 
$i_{L/K}\circ N_{L/K}=[L:K]\,\id _{K_n(K)}$. 
\end{enumerate}

By \cite{Zh1}, $N_{L/F}$ is sequentially 
$P$-continuous and therefore induces a 
continuous morphism of the corresponding 
topological $K$-groups which will be denoted by the same symbol. 

Using the unique extension of $v_K$, define the subgroups 
$U_K^c(K_N^t(L))\subset K_N^t(L)$ 
for all $c\geqslant 0$ and algebraic extensions $L$ of $K$, to be 
the groups generated by the symbols 
$\{\alpha _1,\dots ,\alpha _N\}$ such that 
$v_K(\alpha _1-1)\geqslant c$. Then the general 
definition of the norm map $N_{L/K}$, e.g. cf. \cite{FV}, 
implies that for all $c\geqslant 0$, 
$N_{L/K}$ maps $U^c_K(K_N^t(L))$ to 
$U^c_K(K^t_N(K))$ and preserves the decomposition 
$K_N^t(L)=\Z \oplus VK_N^t(L)$ from Subsection \ref{S1.3}.
\medskip 
\medskip

\section{Pairings in the characteristic $p$ case}\label{S2}

\subsection{Residues}\label{S2.1} 

For any $n\geqslant 0$, denote by $\Omega^n_{L(\c F)}$ 
the $L(\c F)$-module of $P$-continuous differentials 
of degree $n$ for $L(\c F)$. For $n=N$, this module is free of rank 1 with 
the basis $dt_1\wedge \dots \wedge dt_N$. 

Suppose $\omega =fdt_1\wedge\dots \wedge dt_N
\in\Omega ^N_{L(\c F)}$ with $f\in L(\c F)$. Then 
$$f=\sum_{a=(a_0,\dots ,a_N)}[\alpha_a]p^{a_0}t_1^{a_1}\dots t_N^{a_N}$$
and there is an $A_0(f)\in\Z $ such that  
$\alpha _a=0$ if $a_0<A_0(f)$, cf. Subsection \ref{S1.1}. 
This makes sense for the following definition 
of the $L(\c F)$-residue $\Res _{L(\c F)}$ of $\omega $. 

\begin{definition} 
$\Res_{L(\c F)}(\omega )=\sum_{a=(a_0,-1,\dots , -1)}[\alpha _a]p^{a_0}$.
\end{definition}

We have the following standard properties: 

--- if $\omega '\in\Omega ^{N-1}_{L(\c F)}$ then $\Res_{L(\c F)}(d\omega ')=0$;

--- if $\bar t_1',\dots ,\bar t_N'$ is another system of 
local parameters in $\c F$ and $t_1',\dots ,t_N'$ are 
their lifts to $O(\c F)$ then 
$$\Res_{L(\c F)}\left (\frac{dt_1'}{t_1'}
\wedge\dots\wedge\frac{dt_N'}{t_N'}\right )=1.$$

--- if $\Res_{L(\c F)}\omega =c$ then there is an 
$\omega'\in\Omega ^{N-1}_{L(\c F)}$ such that 
$$\omega =d\omega '+c\frac{dt_1'}{t_1'}
\wedge\dots\wedge\frac{dt_N'}{t_N'}$$

The above properties do not show that 
the residue $\Res_{L(\c F)}$ is apriori independent 
of the choice of local parameters of $\c F$ 
because the construction of the 
lift $L(\c F)$ involves a choice of such system 
of local parameters. 
Therefore, we need to slightly modify the above approach 
to the concept of residue.

For any $i\in\Z $,  denote by $O(\sigma ^i\c F)$ 
the $\Z _p$-flat lifts of $\sigma ^i\c F$ via 
the system of local parameters 
$\bar t_1^{p^i},\dots ,\bar t_N^{p^i}$. Set 
$O_M(\sigma ^i\c F):=O(\sigma ^i\c F)/p^M$. 
These  flat $\Z /p^M$-lifts $O_M(\sigma ^i\c F)$ of $\sigma ^i(\c F)$ 
do depend on the system of local parameters 
$\bar t_1,\dots ,\bar t_N$ but we have the following properties:

--- $W_M(\sigma ^{M-1}\c F)\subset O_M(\c F)
\subset W_M(\c F)\subset O_M(\sigma^{-M+1}\c F)$;

--- $W_M(\c F)=O_M(\c F)+pO_{M-1}(\sigma^{-1}\c F)+
\dots +p^{M-1}O_1(\sigma^{-M+1}\c F)$.

Let $\widetilde{\Omega }(\c F,M)$ be the 
$\Z _p$-submodule of $\Omega ^N_{W_M(\c F)}$ 
consisting of differential forms 
$\omega =w\,d_{\log}a_1\wedge\dots \wedge d_{\log}a_N$, 
where $w\in W_M(\sigma ^{M-1}\c F)$ and all 
$a_i\in W_M(\c F)^*$. Then 
$w\in O_M(\c F)$, all $a_i\in O_M(\sigma ^{1-M}\c F)^*$ and, 
therefore, 
$\omega\in\sum _iO_M(\c F)d_{\log }t_i$. As a result, we have a 
natural $W_M(k)$-linear embedding 
$$\iota_{O_M(\c F)}: \wt\Omega(\c F,M) 
\To  O_M(\sigma^{-M+1}\c F)\otimes_{O_M(\c F)}\Omega ^N_{O_M(\c F)}.$$

This means that for any  $\omega\in\tilde\Omega (\c F,M)$, 
the image $\iota_{O_M(\c F)}(\omega)$ can be written 
uniquely as $fd_{\log }t_1\wedge\dots \wedge 
d_{\log}t_N$, where $f=\sum_{a}\gamma _at_1^{a_1}\dots t_N^{a_N}$,  
with the indices $a=(a_1,\dots ,a_N)\in (p^{-M+1}\Z )^N$ 
and the coefficients $\gamma _a\in W_M(k)$. 

\begin{definition} With above notation for any 
$\omega\in\widetilde{\Omega }(\c F,M)$, 
define its $W_M(\c F)$-residue by the relation 
$\Res_{W_M(\c F)}(\omega ):=\gamma _{(0,\dots ,0)}$. 
\end{definition}

This definition is compatible with the 
earlier definition of the $L(\c F)$-residue 
$\Res_{L(\c F)}$ in 
the following sense. If $\omega\in\Omega ^N_{O(\c F)}
\subset\Omega ^N_{L(\c F)}$ and 
$\omega\,\mathrm{mod}\, p^M$  is in  the 
image of $\wt \Omega(\c F,M)$ in 
$O_M(\sigma ^{-M+1}\c F)\otimes_{O(\c F)}
\Omega ^N_{O(\c F)}$ then $\Res _{L(\c F)}
(\omega )\in W(k)$ and 
$\Res_{L(\c F)}(\omega)\,\mathrm{mod}\,p^M=
\Res_{W_M(\c F)}(\omega \mathrm{mod}\,p^M)$. 

We now prove that the $W_M(\c F)$-residue $\Res _{W_M(\c F)}$ 
is independent of the choice 
of local parameters in $\c F$. 
Suppose $\bar t_1',\dots ,\bar t_N'$ is another system 
of local parameters of $\c F$. Consider, for all $i \in \Z$, 
the corresponding flat lifts $O'_M(\sigma ^i \c F)$ and 
the $W_M(\c F)$-residue $\Res'_{W_M(\c F)}$ defined via an analogue 
$\iota_{O'_M(\c F)}$ of $\iota_{O_M(\c F)}$.

\begin{Prop}\label{P2.1} For any 
$\omega\in\widetilde{\Omega }(\c F,M)$, 
one has
$\mathrm{Res} _{W_M(\c F)}(\omega )=
\mathrm{Res}'_{W_M(\c F)}(\omega )$.
\end{Prop}

\begin{proof} 
Note that any $\alpha\in W_M(\c F)^*$ can be written in the 
form $[\beta ]\underline{t}^a\varepsilon\eta $, where 
$[\beta ]\in W_M(k)$ is the Teichm\"uller representative of $\beta\in k$, 
$\underline{t}^a:=t_1^{a_1}\dots t_N^{a_N}$ with $a=(a_1,\dots ,a_N)\in\Z ^N$, 
$\varepsilon\in (1+\m _{L(\c F)})\mathrm{mod}p^MO(\c F)$ 
and $\eta\in 1+pO_{M-1}(\sigma ^{-1}\c F)+
\dots +p^{M-1}O_1(\sigma ^{-M+1}\c F)
=1+pW_{M-1}(\sigma ^{-1}\c F)$. 

With this notation any element of $d_{\log }W_M(\c F)^*$ can be written as
$$\sum _{1\leqslant i\leqslant N}C_i d_{\log }t_i+d_{\log }\varepsilon +d\eta',$$
where $C_1,\dots ,C_N\in\Z $ and $\eta '=\log\eta\in pW_{M-1}(\sigma ^{-1}\c F)$.
Note that the $p$-adic logarithm establishes  an isomorphism 
of the multiplicative group $1+pW_{M-1}(\sigma ^{-1}\c F)$ with 
the additive group $pW_{M-1}(\sigma ^{-1}\c F)$ and 
$$d_{\log }\varepsilon =-\sum _{a}
(\gamma _a\underline{t}^{a})^nd_{\log }\underline{t}^a,$$ 
where $a=(a_1,\dots ,a_n)\in\Z ^N$, $a>\bar 0$, 
$n\geqslant 0$, all $\gamma _{a}\in W_M(k)$ and 
the sum in the right-hand side converges in the $P$-topology. 
(Use that $\varepsilon $ can be written as an infinite product 
$\prod _{a>\bar 0}(1-\gamma _a\underline{t}^a)$.)

Therefore, any element $\omega\in \wt \Omega(\c F,M)$ 
can be written as a sum of the following types of elements:

(i) $\gamma (\omega )d_{\log}t_1\wedge\dots \wedge d_{\log}t_N$ 
with $\gamma (\omega )\in W_M(k)$;

(ii) $md_{\log}t_1\wedge\dots\wedge d_{\log}t_N$ with 
$m\in\m _{L(\c F)}\mathrm{mod}p^MO(\c F)$;

(iii) $d_{\log}t_{i_1}\wedge\dots\wedge d_{\log}t_{i_s}
\wedge d(\eta_1)\wedge\dots\wedge d(\eta_{N-s})$, 
where $0\le s<N$, 
$1\leqslant i_1<\dots <i_s\leqslant N$ and 
$\eta _1,\dots ,\eta _{N-s}\in pW_{M-1}(\sigma ^{-1}\c F)$. 

This follows directly from the above description of the 
elements of $d_{\log}W_M(\c F)^*$ by taking into account 
that for $a\in\Z ^N$ such that $a\geqslant \bar 0$ 
and $\eta\in pW_{M-1}(\sigma ^{-1}\c F)$, we have  
$$(\underline{t}^{a})^n d_{\log}(\underline{t}^a)\wedge d\eta =
\sum_{1\leqslant i\leqslant N}a_id_{\log}t_i\wedge d(\underline{t}^{an}\eta),$$ 
and  $d(w\eta )=wd(\eta )$ for any $\omega\in W_{M-1}(\sigma ^{M-1}\c F)$.

Then it can be seen that $\Res_{W_M(\c F)}(\omega)=
\gamma(\omega)$ by noting that the residues of 
elements of the form (ii)-(iii) are equal to 0.

Finally, it remains to verify that 
$\Res_{W_M(\c F)}(\omega )=1$ for differential 
forms $\omega =d_{\log}t_1'\wedge\dots\wedge d_{\log}t_N'$. 
This can be done also along the 
lines of the above calculations.
\end{proof}

%%%%%%%%%%%%%%%%%%%%%%%%%%%%%%%%%%%%%%%%%%%%%%%%%%%%

\subsection{The Witt symbol} \label{S2.2} 
We introduce the $W_M(\F_p)$-linear pairing 
\begin{equation}\label{E2.2}[\ ,\ \}^{\c F}_M:W_M(\c F)
\times K_N^t(\c F)\longrightarrow W_M(k)
\end{equation}
as follows. 
Suppose $w=(w_0,\dots ,w_{M-1})\in W_M(\c F)$ and 
$\alpha\in K_N(\c F)$ is of the form 
$\alpha = \{\alpha_1,\dots,\alpha_N\} \in K_N(\c F)$. 
For $1\leqslant i\leqslant N$ and $\hat\alpha_i\in W_M(\c F)^*$ 
such that $\hat\alpha_i\,\mathrm{mod}\, p=\alpha_i$, set 
\begin{equation}\label{E2.2a} [w,\alpha \}^{\c F}_M=\Res_{W_M(\c F)}
(\sigma ^{M-1}(w) d_{\log}\hat\alpha _1
\wedge\dots\wedge d_{\log}\hat\alpha _N)\in W_M(k).
\end{equation} 
It can be seen that  
$[w,\alpha\}^{\c F}_M$  is well-defined and 
the pairing it induces factors through 
the natural projection $K_N(\c F)\longrightarrow K_N^t(\c F)$. 

\begin{Lem} \label{L2.2}
For any $w\in W_M(\c F)$ and $\alpha\in K_N^t(\c F)$, 
we have 
$$\sigma [w,\alpha\}^{\c F}_M=[\sigma(w),\alpha\}^{\c F}_M.$$ 
\end{Lem}

\begin{proof} Note that for varying systems of local 
parameters $\bar t_1',\dots ,\bar t_N'$ of $\c F$, the symbols 
$\{\bar t_1',\dots ,\bar t_N'\}$ generate the group $K^t_N(\c F)$. 
Therefore, it is sufficient to consider 
only the symbols $\alpha =\{\bar t_1',\dots ,\bar t_N'\}$. 
By Proposition \ref{P2.1}, the symbol $[\ ,\ \}_M^{\c F}$ 
is independent of the choice of local parameters. 
Therefore, we may assume that 
$\alpha =\{\bar t_1,\dots ,\bar t_N\}$.
Consider the expansion 
$\sigma^{M-1}w=
\sum_a\gamma_a t_1^{a_1}\cdots t_N^{a_N}$, 
where all $\gamma_a \in W_M(k)$, with respect to the 
identification $W_M(\sigma ^{M-1}\c F)= 
O_M(\sigma^{M-1}\c F)+\dots +p^{M-1}O_1(\c F)$. 
Clearly, $[w,\alpha)^{\c F}_M=
\gamma_{(0,\dots ,0)}$ and 
$[\sigma(w),\alpha)^{\c F}_M=\sigma(\gamma_{(0,\dots,0)})$. 
The lemma is proved.
\end{proof}

Notice now that for any $w=(w_1,\dots ,w_M)\in W_M(\c F)$,  
the element $\sigma ^{M-1}w=[w_1]^{p^{M-1}}+p^{M-1}[w_2]^{p^{M-2}}+
\dots +p^{M-1}[w_M]$ 
coincides, modulo $p^M$, with 
the $M$-th ghost component of $w$. 
 Therefore, the classical Witt symbol, cf. \cite{Pa1},
$$[\ ,\ )^{\c F}_M:W_M(\c F)\times K_N^t(\c F)
\longrightarrow W_M(\F_p)$$
has the following invariant form: 
\medskip 

--- {\it If $w\in W_M(\c F)$ and $\alpha\in K_N^t(\c F)$ then 
$[w,\alpha )^{\c F}_M =\mathrm{Tr}\left ([w,\alpha \}^{\c F}_M\right )$.}
\medskip 

Above Lemma \ref{L2.2} implies that the Witt symbol induces 
a $W_M(\F _p)$-linear pairing 
\begin{equation}\label{E2.2b} 
W_M(\c F)/(\sigma -\mathrm{id})W_M(\c F)\ \times\ K_N^t(\c F)
\longrightarrow W_M(\F _p)
\end{equation} 
and it can be verified that 
this pairing is non-degenerate using the explicit formula 
\eqref{E2.2a} for the above symbol $[\ ,\ \}^{\c F}_M$.

%%%%%%%%%%%%%%%%%%%%%%%%%%%%%%%%%%%%%%%%%%%%%%%%%%%%%%%%%%%%%%%%%%%%%%%%%%%%%%%%%%%
\subsection{Coleman's lifts and Fontaine's pairing} \label{S2.3}

For 1-dimensional local fields, Fontaine \cite{Fo} developed a version 
of the Witt symbol by definig a special multiplicative section  
$\mathrm{Col}:\c F^*\longrightarrow O(\c F)^*$ of the natural projection 
$O(\c F)\longrightarrow O(\c F)/p=\c F$. 
 His construction can be generalised in the context of 
topological $K$-groups as follows. 

For any $x\in O_{L(\c F)}$, let $\bar x=(xp^{-v_p(x)})\,\mathrm{mod}\,p\in\c F$. 
Consider the map $\Pi: K_N^t(L(\c F)) \to K_N^t(\c F)$ defined by 
the correspondences 
$$\{x_1,\dots,x_N\} \mapsto 
\{\bar x_1,\dots,\bar x_N\}.$$
We use the free topological generators of 
$K_N^t(\c F)$ from Subsection \ref{S2.2} 
to define the $P$-continuous homomorphism 
$\mathrm{Col}:K_N^t(\c F)\longrightarrow K^t_N(L(\c F))$ by the 
following correspondences: 
$\{\bar t_1,\dots,\bar t_N\}\mapsto \{t_1,\dots,t_N\}$ and 
$\varepsilon _{ja} \mapsto\{1+[\theta _j]\underline{t}^a,t_1,
\dots,t_{i(a)-1},t_{i(a)+1},\dots,t_N\}$. 
This definition makes sense because of the following property, cf. \cite{Zh1}. 

\begin{Lem} \label{L2.3} If for $p$-adic integers $A_{ja}\in\Z _p$, 
where $1\leqslant j\leqslant s$ and 
$a\in\Z ^N\setminus p\Z ^N$, $a>\bar 0$, the product 
 $\prod _{ja}(1+\theta _j\underline{\bar t}^a)^{A_{ja}}$ 
converges in $\c F$ then the product 
$\prod _{ja}(1+[\theta _j]\underline{t}^a)^{A_{ja}}$ 
converges in $L(\c F)$. 
\end{Lem}

The above defined morphism $\mathrm{Col}$ depends on the choice 
of local parameters $\wb t_1,\dots,\wb t_N$ of $\c F$. 
As in Subsection \ref{S2.1} consider the lift $O(\sigma ^{-1}\c F)$. 
Then $L(\sigma ^{-1}\c F)=\mathrm{Frac}\, O(\sigma ^{-1}\c F)$ is 
a field extension of $L(\c F)$ of degree $p^M$. Let $\sigma ^{-1}$ be 
the $\sigma ^{-1}$-linear (with respect to the 
$W(k)_{\Q _p}$-module structure) field isomorphism 
$L(\c F)\longrightarrow L(\sigma ^{-1}\c F)$  
given, for $1\leqslant i\leqslant N$, by the correspondences 
$t_i\longrightarrow t_i^{1/p}$. 

\begin{definition} 
Call an element $x \in K^t_N(L(\c F))$   
Coleman if the norm map from 
$K^t_N(L(\sigma^{-1}\c F))$ to $K^t_N(L(\c F))$ 
maps $\sigma^{-1}(x)$ to $x$. 
\end{definition}

\begin{Prop} \label{P2.4} 
An element $\eta\in K_N^t(L(\c F))$ is Coleman if and 
only if it belongs to $\mathrm{Col}(K_N^t(\c F))$.
\end{Prop}

\begin{proof}
Property (1) of Subsection \ref{S1.4} easily implies that 
all elements from $\mathrm{Col}(K_N^t(\c F))$ are Coleman. 

Suppose $x_0\in K_N^t(L(\c F))$ is Coleman. 
We prove that $x_0\in\mathrm{Col}K_N^t(\c F)$.
Shifting $x_0$ by the inverse to $\mathrm{Col}(\Pi (x_0))$ we may assume that 
$\Pi (x_0)=0$.

Note that $L(\c F)$ has the system of local parameters 
$t_0=p, t_1=\mathrm{Col}(\bar t_1),\dots ,t_N=\mathrm{Col}(\bar t_N)$.
Then the classical identity \eqref{E1.3c} implies that $K^t_N(L(\c F))$ 
is topologically generated by the elements 
$$\{t_0,\dots ,t_{i-1},t_{i+1},\dots ,t_N\}$$ 
with $1\leqslant i\leqslant N$, and the 
elements of the form 
$$\{1+[\theta _j]t_0^{a_0}\dots t_N^{a_N},
\alpha _2,\dots ,\alpha _N\},$$
where $1\leqslant j\leqslant s$, $\theta _1,\dots ,
\theta _s$ is an $\F _p$-basis of $k$, $a=(a_0,\dots ,a_N)
\in\Z ^{N+1}\setminus p\Z ^{N+1}$, 
$a>\bar 0$ and for $2\leqslant i\leqslant N$, 
$\alpha _i=t_{j_i}$ with 
$0\leqslant j_2<j_3<\dots <j_N\leqslant N$. 

These generators can be separated into the two following groups:

--- the first group contains the generators belonging to 
$\mathrm{Ker}\, \Pi $ (in other words these generators do depend on $t_0$);

--- the second group contains the generators from 
$\mathrm{Col}(K_N^t(\c F)).$

Using that $K_N^t(\c F)$ is topologically free, we obtain for any $x\in K_N^t(\c F)$
the following properties:

--- if $\Pi (x)=0$ then $x$ is  
a product of generators from the first group;

--- if $\Pi (x)=0$ and $x=p^mx_1$ with $m\geqslant 0$ 
and $x_1\in K_N^t(L(\c F))$, then $\Pi (x_1)=0$.

Returning to the Coleman element $x_0\in K_N^t(L(\c F))$, assume that 
there is an $m\geqslant 0$ such that $x_0=p^mx_1$ with 
$x_1\in K_N^t(L(\c F))$ but $x\notin pK_N^t(L(\c F))$.
Then $\Pi (x_1)=0$ and $x_1$ is a product of generators 
from the first group. But if $y$ is a generator from this group then 
property 1) of Subsection \ref{S1.4} implies that 
$N_{L(\sigma ^{-1}\c F)/L(\c F)}(\sigma ^{-1}y)\in pK_N^t(L(\c F))$. 
This gives that 
$x_0=N_{L(\sigma ^{-1}\c F)/L(\c F)}(\sigma ^{-1}x_0)=
p^mN_{L(\sigma ^{-1}\c F)/L(\c F)}(\sigma ^{-1}x_1)$ 
$\in p^{m+1}K_N^t(L(\c F))$, which is a  
contradiction. This means that $x_0$ is infinitely $p$-divisible and, 
therefore, is $0$ in $K_N^t(L(\c F))$.
\end{proof}

We define an analogue of Fontaine's pairing \cite{Fo} 
$$[\ ,\ \rangle^{\c F}: O(\c F) 
\times K_N^t(\c F) \longrightarrow \Z_p$$ 
by  setting for $f \in O(\c F)$ and 
$\alpha\in K_N(\c F)$,
$$[f,\alpha\rangle ^{\c F}=
\Tr\big(\Res_{L(\c F)} f d_{\log} \mathrm{Col}(\alpha)\big).$$ 
Here, for $\mathrm{Col}(\alpha)=\{\hat\alpha_1,\dots,\hat\alpha_N\}$, we set  $d_{\log}\mathrm{Col}(\alpha)=d_{\log}\hat\alpha_1
\wedge\dots\wedge d_{\log}\hat\alpha_N$. 
This pairing is related to the Witt symbol by the following Proposition. 
\begin{Prop} \label{P2.5}
For all $f \in O(\c F)$ and $\alpha \in K_N^t(\c F)$, one has 
$$[f,\alpha\rangle^{\c F}\mathrm{mod}\,p^M=
[f\,\mathrm{mod}\,p^M,\alpha\,\mathrm{mod}\,p^M)^{\c F}_M.$$
\end{Prop}
\begin{proof}
We need to show that 
\begin{equation}\label{E2.3}\sigma ^{M-1}
\Res _{L(\c F)}(fd_{\log }\mathrm{Col}\alpha )=
\Res _{L(\c F)}(\sigma ^{M-1}(f)d_{\log }\mathrm{Col}(\alpha )).
\end{equation}

By linearity and $P$-continuity this can be verified on the 
generating elements 
$f=\underline{t}^{b}=t_1^{b_1}\dots t_N^{b_N}$, 
$b=(b_1,\dots ,b_N)\in\Z ^N$, of $O(\c F)$ and the generators 
$\alpha =\{\bar t_11,\dots ,\bar t_N\}$ and 
$\alpha =\varepsilon _{ja}$ of $K_N^t(\c F)$ 
from \eqref{E2.3} of Subsection \ref{S1.3}.

--- {\it The case $f=\underline{t}^b$ and $\alpha =
\{\bar t_1,\dots ,\bar t_N\}$}. In this case 
the both sides of equality \eqref{E2.3} are equal to the Kronecker 
symbol $\delta (b,\bar 0)$.

--- {\it The case $f=\underline{t}^b$ and $\alpha =\varepsilon _{ja}$}. 
Here the left-hand side of 
\eqref{E2.3} equals 
$$(-1)^{i(a)-1}\sigma ^{M-1}\Res _{L(\c F)}
\left (\underline {t}^b\sum _{n\geqslant 0}
(-1)^n[\theta _j^n]\underline{t}^{na}
\frac{dt_1}{t_1}\wedge\dots\wedge\frac{dt_N}{t_N}\right )$$
and the corresponding right-hand side equals 
$$(-1)^{i(a)-1}\Res _{L(\c F)}
\left (\underline {t}^{p^{M-1}b}
\sum _{n\geqslant )}(-1)^n[\theta _j^n]\underline{t}^{na}
\frac{dt_1}{t_1}\wedge\dots\wedge\frac{dt_N}{t_N}\right ).$$

Clearly, we may assume that $b\ne\bar 0$. Then the left-hand side 
is non-zero if and only if there is an $n_0\geqslant 1$ 
such that $b+n_0a=\bar 0$. This is equivalent to saying 
that the right-hand side is non-zero noting that  
$a\not\in p\Z ^N$. It can then be seen that 
 both sides are equal to 
$(-1)^{i(a)+n_0-1}a_{i(a)}[\theta _j^{p^{M-1}n_0}]$.
\end{proof}

%%%%%%%%%%%%%%%%%%%%%%%%%%%%

\section{Vostokov's pairing}\label{S3} 

As usual, $\pi _1,\dots ,\pi _N$ is a fixed system of local parameters  
 and $k$ is the $N$-th residue field of $F$. 
Let $L_0(\c F)=W(k)((t_N))\dots ((t_1))\subset L(\c F)$ with the induced 
topological structure. Set 
$$\m ^0=\left\{\sum_{a>\bar 0}w_a\underline{t}^a\ |\ w_a\in W(k)\right\}$$
and $O^0=W(k)+\m ^0$. Clearly, 
$\m ^0\subset O^0\subset L_0(\c F)$ and $L_0(\c F)=
\underset{a>\bar 0}\bigcup \underline{t}^{-a}O^0$. 
Let $\c R$ be the multiplicative subgroup in 
$L_0(\c F)^*$ generated by the Teichmuller representatives of the elements of $k$, 
the indeterminants $t_1,\dots ,t_N$ and the elements of $1+\m ^0$. 
Let $\kappa :\c R\longrightarrow F$ be the epimorphic 
continuous morphism of $W(k)$-algebras such that 
$\kappa (t_i)=\pi _i$, where $1\leqslant i\leqslant N$. 
We use the same notation $\kappa $ for the unique 
$P$-continuous epimorphism of $W(k)$-algebras 
$L_0(\c F)\longrightarrow F$ such that $\kappa (t_i)=
\pi _i$ for $1\leqslant i\leqslant N$.

\subsection{The differential form $\Omega $} \label{S3.1}
For any $u_0,\dots ,u_N\in \c R$, denote by 
$\Omega =\Omega (u_0,\dots ,u_N)$ the following differential 
form from $\Omega ^N_{L_0(\c F)}$:
$$\underset{0\leqslant i\leqslant N}\sum (-1)^if_i
\left (\frac{\sigma }{p}d_{\log }u_0\right )\wedge\dots\wedge
\left (\frac{\sigma }{p}d_{\log}u_{i-1}\right )\wedge d_{\log}u_{i+1}
\wedge\dots\wedge d_{\log}u_N$$
where for $0\leqslant i\leqslant N$, $f_i=(1/p)\log (u_i^p/\sigma u_i)$. 
Notice that all $f_i\in\m ^0$ (use that 
$\sigma u_i/u_i^p\in 1+\m ^0$) 
and 
\begin{equation} \label{E3.1}
df_i=d_{\log}u_i-(\sigma /p)d_{\log }u_i. 
 \end{equation}

\begin{Prop} \label{*3.1} $\Omega\,\mathrm{mod}\,
d\Omega ^{N-1}_{O^0}$ is skew symmetric in 
 $u_0,\dots ,u_N$.
\end{Prop} 

\begin{proof} Prove that 
$\Omega\,\mathrm{mod}\,d\Omega _{O^0}^{N-1}$ changes the sign 
under the transpositions $u_i\leftrightarrow u_{i+1}$, $0\leqslant i<N$. 
Consider the identity (use \eqref{E3.1})
$$f_id_{\log}u_{i+1}-f_{i+1}(\sigma /p)d_{\log }u_i+
f_{i+1}d_{\log }u_i-f_i(\sigma /p)d_{\log }u_{i+1}=d(f_if_{i+1})$$
Then the form 
$\Omega (\dots ,u_i,u_{i+1},\dots )+\Omega (\dots ,u_{i+1},u_i,\dots )$ 
is congruent modulo $d\Omega ^{N-1}_{O^0}$ to the form 
$$(\sigma /p)d_{\log }u_0\wedge\dots\wedge
(\sigma /p)d_{\log}u_{i-1}\wedge 
d(f_if_{i+1})\wedge d_{\log}u_{i+1}
\wedge\dots\wedge d_{\log}u_N$$
and, using again identity \eqref{E3.1}, we conclude that 
this form is exact. 
\end{proof}

Let $e=(e_1,\dots ,e_N)\in\Z ^N$ be such that 
$\pi _1^{e_1}\dots \pi _N^{e_N}/p\in\c O_F^*$, where $\c O_F$ is the 
$N$-dimensional valuation ring of $F$. We introduce 
the $W(k)$-algebra 
$\c L^0=O^0[[p/\underline{t}^{e(p-1)},\underline{t}^{ep}/p]]$
and set $\c L=\c L^0\otimes _{O^0}L_0(\c F)$. 
Clearly, we have $\c L=\underset{a>\bar 0}\bigcup \underline{t}^{-a}\c L^0$.

The algebra $\c L$ is a suitable completion of 
$L_0(\c F)$ and its elements can be treated as formal 
Laurent series in $t_1,\dots ,t_N$ with coefficients in $W(k)$. In particular,  
 $\c L^0$ consists of  
formal Laurent series $\sum _{a\in\Z ^N}w_a\underline t^a$ 
with coefficients $w_a\in W(k)$, such that: 

--- if $n\geqslant 0$ and $a\geqslant epn$ then $v_p(w_a)\geqslant -n$;

--- if $n\geqslant 0$ and $a\geqslant -ep(n-1)$ then $v_p(w_a)\geqslant n$.
 
We can use the above Laurent series to define 
the $\c L$-residues $\Res _{\c L}\,\omega $ for any 
$\omega\in \Omega ^N_{\c L}$. 
If any such form $\omega $ is the limit 
of $\omega _n\in\Omega ^N_{L_0(\c F)}$, then 
$\Res _{\c L}\,\omega $ is the limit of $\Res \,\omega _n$. 
Therefore, we can use  
for the $\c L$-residue of $\omega $, 
the simpler notation $\Res\,\omega $. 

\begin{Lem}\label{L3.2}
 Let $\pi _1^{e_1}\dots \pi _N^{e_N}/p=\eta \in \c O^*_F$ 
and let $\hat\eta\in O^0$ be such that $\kappa (\hat\eta )=\eta $. 
Then the kernel of $\kappa : O^0(\c F)\longrightarrow F$ is the principal 
ideal generated by $\underline{t}^e-p\hat\eta $.
\end{Lem}

\begin{proof} The proof follows easily from the fact that 
 $\kappa $ induces a bijective map 
$O^0(\c F)/\underline{t}^e\rightarrow \c O_F/p$.
\end{proof}

\begin{Prop} \label{P3.3} If $u_0\in \c R$ and $\kappa (u_0)=1$ then there are 
$\omega ^0,\omega ^1\in\Omega ^{N-1}_{O^0}$ 
such that for 
$\Omega '=\log(u_0)d_{\log }u_1
\wedge\dots\wedge d_{\log}u_N-\frac{\sigma }{p}\log(u_0)\frac{\sigma }{p}d_{\log }u_1\wedge\dots \wedge 
\frac{\sigma}{p}d_{\log }u_N $, it holds 
$\Omega =\Omega '+d(\log(u_0)\omega ^0+\omega ^1)$. 
\end{Prop}

\begin{proof} Clearly, $u_0\in 1+\m ^0$. 
By above Lemma \ref{L3.2} the relation $\kappa (u_0)=1$ 
 implies that $\log(u_0)\in\c L^0$. Then the statement of our Proposition 
is implied by the following identities: 
$f_0=\log(u_0)-(\sigma /p)\log(u_0)$ and for $1\leqslant i\leqslant N$, 
$$f_i(\sigma /p)d_{\log}u_0=d(f_i\log(u_0)-f_if_0)-(\sigma /p)
\log(u_0)(d_{\log}u_i-(\sigma /p)d_{\log}u_i).$$
\end{proof}

\subsection{Element $H_0$.} \label{S3.2} 
As in the Introduction, choose a primitive $p^M$-th root of unity 
$\zeta _M\in F$ and introduce $H_0\in \m ^0$ such that 
$H_0=H^{\prime p^M}-1$, where $H'\in 1+\m ^0$ is such that 
$\kappa (H')\equiv\zeta _M\,\mathrm{mod}\,pO_F$. 

Clearly, we have $dH_0\in p^M\Omega ^1_{O^0}$. 

\begin{Lem} \label{L3.4} a) There are $o_1\in O^{0*}$ and 
$o_2,o_3\in O^0$ such that 

a) $H_0=o_1\underline{t}^{ep/(p-1)}+po_2
\underline{t}^{e/(p-1)}+p^2o_3$;

b) $H_0^{-1}\in\underline{t}^{-ep/(p-1)}
O^0[[p\underline{t}^{-e}]]\subset\c L$; 

c) $H_0^{p-1}/p\in O^0[[\underline{t}^{ep}/p]]\subset\c L^0$ and 
$O^0[[H_0^{p-1}/p]]=O^0[[\underline{t}^{ep}/p]]$;

d) $(\sigma /p)H_0=H_0(1+o_1H_0+o_2(H_0^{p-1}/p)+o_3(p^M/H_0))$,  
where the coefficients $o_1,o_2,o_3\in O^0$ . 
\end{Lem}

\begin{proof} In order to prove a) use that 
 $H'\equiv 1+o\underline{t}^{e/p^{M-1}(p-1)}
\,\mathrm{mod}\,(p,\underline{t}^e)$ with $o\in O^{0*}$. 
Then b) and c) are implied by a). For part d), use that 
$\sigma H'\equiv H^{\prime p}\,\mathrm{mod}\,pO^0$ and 
therefore, $\sigma H_0\equiv (1+H_0)^p-1\,\mathrm{mod}\,p^{M+1}O^0$.
\end{proof}

\begin{Lem} \label{L3.5} 
If $\omega =\log(u_0)\omega _1+\omega _0$ with 
$\omega _0,\omega _1\in \Omega ^{N-1}_{O^0}$ then we have    
 $\Res (H^{-1}d\omega )\in p^MW(k)$.
\end{Lem}

\begin{proof} Note that $\Res (H^{-1}d\omega )=-\Res (d(1/H)\wedge \omega )$, 
where $\omega /H^2\in\underline{t}^{-2ep/(p-1)}\c L^0\Omega _{O^0}^{N-1}$ and 
$dH\in p^M\Omega ^1_{O^0}$. It follows that 
$$d(1/H)\wedge \omega =(\omega /H^2)\wedge dH\in p^M\underline{t}^{-2ep/(p-1)}
\c L^0\Omega ^N_{O^0}.$$
It remains to notice that 
$ep>2ep/(p-1)$, which implies that 
\linebreak 
$\Res (t^{-2ep1/(p-1)}\c L^0dt_1\wedge\dots\wedge dt_N)\subset W(k)$.
\end{proof}

\begin{Cor}\label{C3.6}  With the notation from Proposition {\rm\ref{P3.3}}, we have  
$$\Res (\Omega /H)\equiv\Res (\Omega '/H)\,\mathrm{mod}\,p^M.$$
\end{Cor}

\begin{Prop} \label{P3.7} 
If $h\in O^0[[\underline{t}^{ep}/p]]$ then 
$$\Res \left (\frac{hd_{\log}t_1\wedge\dots\wedge d_{\log}t_N}{H}\right )
\equiv\Res \left (\frac{hd_{\log}t_1\wedge\dots\wedge d_{\log}t_N}{(\sigma /p)H}\right )
\mathrm{mod}\,p^M.$$
\end{Prop}

\begin{proof} We follow the strategy from 
the proof of Lemma 3.1.3 in [Ab]. 

By Lemma \ref{L3.4}d) it will be sufficient to prove 
the congruence 
\linebreak 
$\Res (\c Gd_{\log}t_1\wedge\dots \wedge d_{\log}t_N)
\equiv 0\,\mathrm{mod}\,p^M$, where 
$$\c G=\frac{h_1H_0^{l_1}
(H_0^{p-1}/p)^{l_2}(p^M/H)^{l_3}}{H_0}$$
with $h_1\in O^0[[\underline{t}^{ep}/p]]$, 
$l_1,l_2,l_3\in\Z _{\geqslant 0}$ and $l_1+l_2+l_3\geqslant 1$. 

If $l_3=0$ then 
$\c G\in\c O^0[[\underline{t}^{ep}/p]]$ 
and, therefore, the residue 
\linebreak 
$\Res (\c Gd_{\log}t_1\wedge\dots\wedge d_{\log}t_N)=0$.

If $l_3\geqslant 1$ then use that $H^{p-1}/p, p/H\in\c L^0$ 
to obtain that 
$$\c G\in (p^M/H^2)\c L^0\subset p^M\underline{t}^{-2ep/-1)}\c L^0.$$
Similarly to the proof of Lemma \ref{L3.5} this implies that 
the residue $\Res (\c Gd_{\log}t_1\wedge\dots\wedge d_{\log}t_N)\in p^MW(k)$. 
\end{proof}

\subsection{A construction of Vostokov's pairing $V$.} 
\label{S3.3} 
For any elements  $u_0,\dots ,u_N\in \c R$, set 
\begin{equation}\label{E3.3}
\widetilde{V}(u_0,\dots ,u_N)=\Tr \left (\Res 
\frac{\Omega }{H}\right )\mathrm{mod}p^M, 
\end{equation}
where, as earlier, $\Tr $ is the trace map 
for the field extension $W(k)_{\Q _p}/\Q _p$. 
Then Proposition \ref{*3.1} and Lemma \ref{L3.5} imply that $\widetilde{V}$ is 
an $(N+1)$-linear skew-symmetric form on
$\c R$ with values in $\Z /p^M$. 

\begin{Prop} If $\kappa (u_0)=1$ then 
 $\widetilde{V}(u_0,\dots ,u_N)=0$.
\end{Prop}

\begin{proof} By Propositions 3.3 and 3.7 
it will be sufficient to prove that 
$$\sigma\Res\left (\frac{\log(u_0)}{H_0}d_{\log}
u_1\wedge\dots\wedge d_{\log}u_N\right )\equiv 
\qquad\qquad\qquad \qquad\qquad\qquad $$
$$\qquad \Res \left (\frac{(\sigma /p)\log(u_0)}{(\sigma /p)H_0}
(\sigma /p)d_{\log}u_1\wedge\dots \wedge 
(\sigma /p)d_{\log}u_N\right )\mathrm{mod}\,p^M$$
 Let $\c L^{(-1)}$ be the subalgebra in $\c L$ 
consisting of the formal  Laurent series 
$l=\sum _{a\in\Z ^N}w_a\underline{t}^a$ such that 
$\sigma (l):=\sum _{a\in\Z ^N}\sigma (w_a)\underline{t}^{ap}\in\c L$. 
Then one can verify that for any $r\in\c L^{(-1)}$, 
$$\sigma\Res (rd_{\log}u_1\wedge\dots\wedge d_{\log}u_N)=
\Res \left (\sigma (r)\frac{\sigma }{p}d_{\log}u_1
\wedge\dots\wedge\frac{\sigma }{p}d_{\log}u_N\right ).$$
It remains to note that $\kappa (u_0)=1$ 
implies (use Lemma \ref{L3.2}) 
that $r=\log(u_0)/H_0\in\c L^{(-1)}$ and, therefore, 
$\sigma (r)=(\sigma /p)\log(u_0)/(\sigma /p)H_0$. 
\end{proof}

\begin{Cor} The form $\widetilde{V}$ 
factors through the projection 
$\kappa :\c R\longrightarrow F^*$ and defines an 
$(N+1)$-linear skew-symmetric form $\bar V$ on $F^*$ 
with values in $\Z /p^M$. 
\end{Cor}

We now verify the Steinberg relation for $\widetilde{V}$. 

\begin{Prop} If $u_1+u_0=1$ then $\widetilde{V}(u_0,u_1,\dots u_N)=0$. 
 \end{Prop} 

\begin{proof} As usually, it is sufficient to verify this property 
for $u_0\in \m ^0$. 
Then by Lemma \ref{L3.5}  
it will be sufficient to prove that 
\begin{equation}\label{E3.3a}
f_0d_{\log}u_1-f_1(\sigma /p)d_{\log}u_0=dF
\end{equation}
where $F\in O ^0$.

For any $u\in \c R$, set $l(u):=(1/p)\log (u^p/\sigma u)$. 
By computing in $L_0(\c F)\otimes\Q _p$ we obtain the identity 
$$l(u_0)d_{\log}(1-u_0)=
d(\Li _2(u_0)+\log (1-u_0)l(u_0))
$$
where $\Li _2(X)=\int\log (1-X)X^{-1}dX=
\sum _{n\geqslant 1}X^n/n^2$ is the dilogarithm function.
This identity implies that \eqref{E3.3a} holds with 
$$F=\Li _2(u_0)-(1/p^2)\Li _2(\sigma u_0)+\log (1-u_0)l(u_0).$$

It remains to prove that $F\in O ^0$. 

Using the expansions for $\Li _2(X)$ and $\log (1-X)$ we can 
rewrite $F$ as a double sum $F=\sum _{m,s}F_{ms}u_0^{mp^s}$, 
where:

--- the indices $s$ and $m$ run over the set of all 
non-negative integers with additional condition that 
$m$ is prime to $p$;

--- for all (prime to $p$) indices $m$, we have 
$$F_{m0}=1/m^2-(1/m)l(u_0)$$ 
and for all $s\geqslant 1$, 
$$F_{ms}=\frac{1}{m^2p^{2s}}
\left (1-\frac{\sigma u_0^{mp^{s-1}}}{u_0^{mp^s}}\right )-
\frac{1}{mp^s}l(u_0).$$
Clearly, $F_{m0}\in O^0$ and $F_{ms}$ appears as 
the result of the substitution of $X=-ml(u_0)\in\m ^0$ into 
the $p$-integral power series of the function 
$(p^sX)^{-2}(1+p^sX-\exp(p^sX))$. 
Therefore, all $F_{ms}\in O^0$ and $F\in O ^0$. 
\end{proof}

\begin{Cor} \label{C3.11} 
 $\widetilde{V}$ induces a bilinear continuous non-degenerate pairing 
$V:K_1(F)/p^M\times K_N(F)/p^M\longrightarrow \Z /p^M$, 
which factors through the canonical morphism 
of the left-hand side to $K_{N+1}(F)/p^M$. 
\end{Cor}

\begin{proof} The only thing to verify is non-degeneracy. 
 This can be done by routine calculations with the corresponding 
topological generators from Subsections \ref{S1.2} and \ref{S1.3}. 
The most important fact is that 
$$V(\epsilon _0, \{\pi _1,\dots ,\pi _N\})=1,$$
 where 
$\epsilon _0=\theta (\alpha _0H_0)$ is the $p^M$-primary element 
from Remark 1) of Subsection \ref{S1.2}.  
\end{proof}

\begin{remark} 1) The above construction of the pairing $V$ 
depends on the choice of a primitive $p^M$-th root of unity 
$\zeta _M$. However, Vostokov's pairing 
appears in the form 
\begin{equation}\label{E:Vost}
\c V:K_1(F)/p^M\times K_N(F)/p^M\longrightarrow \langle \zeta _M\rangle .
 \end{equation} 
where for any $\alpha\in K_1(F)/p^M$ and $\beta\in K_N(F)/p^M$, 
$\c V(\alpha ,\beta )=\zeta _M^{V(\alpha ,\beta )}$, 
and this pairing is independent of the choice of $\zeta _M$.

2) The above Corollary immediately implies that 
Vostokov's pairing \eqref{E:Vost} coincides with the $M$-th Hilbert symbol 
$$(\ ,\ )_M:K_1(F)/p^M\times K_N(F)/p^M
\longrightarrow \langle\zeta _M\rangle $$
for the field $F$. 
Indeed, the norm property of the Hilbert symbol implies that it factors 
through the canonical morphism $K_1(F)/p^M\times K_N(F)/p^M$ 
to $K_{N+1}(F)/p^M$. Therefore, it is sufficient to verify 
that the Hilbert pairing is equal to $\zeta _M$ on the generator  
$\{\varepsilon _0,\pi _1,\dots ,\pi _n\}$ of 
$K_{N+1}(F)/p^M$. But this is exactly the 
basic property of the $p^M$-primary element 
$\theta (\alpha _0H_0)$, cf. Subsection \ref{S5.5}. 
\end{remark}

%%%%%%%%%%%%%%%%%%%%%%%%%%%%%%%%

\section{The field-of-norms functor}\label{S4}

%%%%%%%%%%%%%%%%%%%%%%%%%%%%%%%%%%

%%%%%%%%%%%%%%%%%%%%%%%%%%%%%%%%%%%%%%%%%%%%%%%%%%%%%%%%%%%%%%%%%%%%%%%%%%%%%%%%%%%
\subsection{The field-of-norms functor}\label{S4.1}

$N$-dimensional local fields are  
special cases of the $(N-1)$-big fields 
used in \cite{Sch}  
to construct a higher dimensional analogue of 
the field-of-norms functor. 
The main ideas of this construction are as follows. 

Suppose $K$ is an $N$-dimensional local field and 
$v_K:K\longrightarrow 
\Z\cup\{\infty \}$ is the (first) valuation of $K$. 
If $\bar K$ is an algebraic closure of $K$, 
denote by the same symbol the unique 
extension of $v_K$ to $\bar K$. 
For any $c\geqslant 0$, let $\p _K^c=
\{x\in\bar K\ |\ v_K(x)\geqslant c\}$. If $L$ is a 
field extension of $K$ in $\bar K$, we use 
the simpler notation $\c O_L/\p ^c$ instead of 
$\c O_L/(\p^c\cap\c O_L)$.
Clearly, if $[L:K]<\infty $ and $e(L/K)$ 
is the ramification index of $L/K$, 
then  $\p ^c_K=\p ^{ce(L/K)}_L$.

An increasing fields tower $K_\d=(K_n)_{n\geqslant 0}$, where $K_0=K$, 
is strictly deeply ramified (SDR) with parameters $(n_0,c)$,  if for 
$n\geqslant n_0$, one has 
$[K_{n+1}:K_n]=p^N$, $c\leqslant v_K(p)$  and 
there is a surjective map 
$\Omega^1_{\O_{F_n+1}/\O_{F_n}} \To (\O_{F_{n+1}}/\p ^c)^d$ 
or, equivalently, 
the $p$-th power map induces epimorphic maps 
\begin{equation} \label{E4.1}
i_N(K_\d ):\c O_{K_{n+1}}/\p _K^c\longrightarrow\c O_{K_n}/\p ^c_K.
\end{equation}
This means that for $n\geqslant n_0$, $K_{n+1}^{(N)}=K_n^{(N)}$ and 
there are systems of local parameters $u_1^{(n)},\dots ,u_N^{(n)}$ in 
$K_n$ such that for all $1\leqslant i\leqslant N$, 
$u_i^{(n+1)p}\equiv u_i^{(n)}
\,\mathrm{mod}\,\p_K^c$.

The field-of-norms functor $X$ associates to any SDR tower $K_\d $ 
an $N$-dimensional field $\c K=X(K_{\d })$ of characteristic $p$ such that its 
$N$-dimensional valuation ring $\c O_{\c K}$ coincides with 
$\underset{i_n(K_{\d })}\varprojlim\c O_{K_n}/\p _K^c$. 
Then for $n\geqslant n_0$, we have the following properties:
\medskip 

--- the last residue fields of $\c K$ and $K_n$ coincide;
\medskip 

--- the natural projections from $\c O_{\c K}$ to 
$O_{K_n}/\p _K^c$ induce isomorphisms of unitary rings 
\begin{equation}\label{E4.1a}
\c O_{\c K}/\p _{\c K}^{c_n}\simeq\c O_{K_n}/\p _K^c
 \end{equation} 
where $c_n=p^{n-n_0}e(K_{n_0}/K)$;
\medskip 

--- if $\bar u_1,\dots ,\bar u_N$ is a system of local parameters in $\c K$ then there are 
systems of local parameters $u_1^{(n)},\dots ,u_N^{(n)}$ in $K_n$ such that 
for $1\leqslant i\leqslant N$, $\underset{n}\varprojlim\, u_i^{(n)}=\bar u_i$.
\medskip 

Suppose $L$ is a finite extension of $K$ in $\bar K$. Then the tower 
$L_\d =(LK_n)_{n\geqslant 0}$ is again SDR and $X(L_{\d })=\c L$ 
is a separable extension of $\c K$ of degree $[LK_n:K_n]$, 
where $n\gg 0$. The extension $\c L/\c K$ is Galois if and only if,  
for $n\gg 0$, $LK_n/K_n$ is Galois.  An analogue of the identification 
\eqref{E4.1a} can be used to identify 
$\Gal (\c L/\c K)$ with $\Gal (LK_n/K_n)$. 

Finally, $X(\bar K):=\underset{\substack{\longrightarrow\\L}}\lim X(L_{\d})$ 
is a separable closure $\c K_{sep}$ of $\c K$ and the functor 
$X$ identifies 
$\Gal (\c K_{sep}/\c K)$ with $\Gal (\bar K/K_{\infty })$, 
where $K_{\infty }=\underset{\substack{\longrightarrow \\ n}}\lim K_n$.

\subsection{Applications to $K$-groups.}\label{S4.2}

Suppose there is an SDR tower $K_{\d }=(K_n)_{n\geqslant 0}$ with 
parameters $(n_0,c)$ and the ring epimorphisms $i_n=i_n(K_{\d }):
\c O_{K_{n+1}}/\p ^c_{K}\longrightarrow \c O_{K_n}/\p ^c_K$. 
Define for $n\geqslant n_0$, the morphisms 
$$j_n=j_n(K_{\d }):K_N^t(K_{n+1})/U^c_KK_N^t(K_{n+1})
\longrightarrow K_N^t(K_n)/U_K^cK_N^t(K_n).$$
as follows. 
Choose systems of local parameters 
$u_1^{(n)},\dots ,u_N^{(n)}$ of $K_n$ such that 
for $1\leqslant i\leqslant N$, 
$u_i^{(n+1)p}\equiv u_i^{(n)}\,\mathrm{mod}\,\p ^c_K$. 
Consider the lifts $\hat\imath _n$ of the morphisms $i_n$:
$$\hat\imath _n:\sum _{a}[\theta _a]\underline{\pi }^{(n+1)a}
\mapsto\sum _{a} [\theta _a^p]\underline\pi ^{(n)a},$$
where $a=(a_1,\dots ,a_N)\in\Z ^N$, $\theta _a\in K^{(N)}$, and 
$\underline{\pi }^{(n)a}=\pi _1^{(n)a_1}\dots \pi _N^{(n)a_N}$. 
Then $\hat\imath _n(K_{n+1}^*)\subset K_n^*$ and we may  
consider 
$$\hat \imath_n^{N}:K_{n+1}^{*\otimes N}\longrightarrow K_N^t(K_n)$$
induced by  
$\alpha _1\otimes\dots \otimes\alpha _N
\mapsto\{\hat\imath _n(\alpha _1),\dots ,\hat\imath _n(\alpha _N)\}$. 
Note that for any $\alpha ,\beta\in K_{n+1}^*$, we have  

--- $\hat\imath _n(\alpha\beta )=\hat\imath _n(\alpha )\hat\imath _n(\beta )
\gamma _1$ with $\gamma _1\in\p _K^c$;

--- $\hat\imath _n(\alpha +\beta )=\hat\imath _n(\alpha )\gamma _2+
\hat\imath _n(\beta )$ with $v_K(\alpha )\leqslant v_K(\beta )$ and 
$\gamma _2\in\p _K^c$.

This implies that the images of 
the relations for the Milnor $K$-group $K_N(K_{n+1})$ 
lie in $U^c_KK_N^t(K_n)$. Therefore, 
$\hat\imath _n^N$ factors through the natural projection to 
$K_N^t(K_{n+1})$ and we can proceed modulo 
the subgroup 
$U^c_KK_N^t(K_{n+1})$ because it is mapped to 
$U^c_KK_N^t(K_n)$. It can be seen that the morphisms 
$j_n$, $n\geqslant n_0$, depend only on the tower $K_{\d }$ 
and its parameters $(n_0,c)$.

Note that 
\begin{equation}\label{E4.2}
(K_N^t(K_n)/U^c_KK_N^t(K_n),j_n)_{n\geqslant n_0}
\end{equation} 
is a projective system. 
Consider the system of local parameters 
\linebreak 
$\bar u_1=\underset{n}\varprojlim\,u _1^{(n)},\dots ,\bar u_N=
\underset{n}\varprojlim\,u _N^{(n)}$ for 
 $\c K=X(F_{\d })$. Similarly to the procedure of 
constructing the morphisms $j_n$, we use the identifications 
$\c O_{\c K}/\p ^{c_n}_{\c K}\simeq\c O_{K_n}/\p ^c_K$ 
to construct  isomorphisms of topological $K$-groups 
$$\iota _n:K_N^t(\c K)/U^{c_n}_{\c K}K_N^t(\c K)\longrightarrow K_N^t(K_n)/U_K^cK_N^t(K_n).$$
These morphisms are compatible with the morphisms $j_n$ 
introduced above 
and determine the isomorphism
$$\iota ^{(c)}:K_N^t(\c K)\longrightarrow
\underset{n}\varprojlim K_N^t(K_n)/U^c_KK_N^t(K_n).$$

For any $0<c'<c$, there is an induced 
projective system 
\begin{equation}\label{E4.2a}
(K_N^t(K_n)/U^{c'}_KK_N^t(K_n),j'_n)_{n\geqslant n_0}.
\end{equation} 
Its inverse limit coincides with $K_N^t(\c K)$ and the composition 
of $\iota ^{(c)}$ with the natural projection from 
\eqref{E4.2} to \eqref{E4.2a} coincides with $\iota ^{(c')}$. 
\medskip 

\subsection{Special SDR towers}\label{S4.3}

We need the following additional 
assumption about the SDR towers $K_{\d }$.

\begin{definition} An SDR tower $K_{\d}$ will be 
called special if for any $n\geqslant n_0$, 
there is a fields tower of extensions of relative degree $p$ 
$$K_n=K_n^0\subset K_n^1\subset\dots
\subset K_n^{N-1}\subset K_n^N=K_{n+1}.$$
\end{definition} 

Our main applications are related to the 
following example of a special SDR-tower.

\begin{definition} The tower $F_{\d }^0=(F^0_n)_{n\geqslant 0}$ 
is very special if $F^0_0=F$ and for all $n\geqslant 0$, the 
field $F^0_{n}$ has a system of local parameters 
$\pi ^{(n)}_1,\dots \pi ^{(n)}_N$ such that for $1\leqslant i\leqslant N$, 
$\pi _i^{(n+1)p}=\pi _i^{(n)}$.
\end{definition} 

Clearly, a very special tower $F^0_{\d }$ is SDR with parameters $(0,v_F(p))$ 
and $X(F^0_{\d })=\c F$. 
In this case we have also an isomorphism 
$K_N(\c F)\simeq\underset{n}\varprojlim K_N^t(F^0_n)$ 
induced in terms of generators 
from Subsection \ref{S1.3} 
by the morphisms $K^t_N(\c F)\longrightarrow K_N^t(F^0_n)$ 
such that:
\medskip 

--- $\{\bar t_1,\dots ,\bar t_N\}
\mapsto \{\pi _1^{(n)},\dots ,\pi _N^{(n)}\}$;
\medskip 

--- $\{1+[\theta _j]\underline{\bar t}^{\ a},\bar t_1,\dots ,\bar t_{i(a)-1},
\bar t_{i(a)+1},\dots ,\bar t_N\}\mapsto $ \newline 
$$\{ 1+[\theta ^{p^{-n}}_j]\underline{\pi }^{(n)a},
\pi ^{(n)}_1,\dots ,\pi ^{(n)}_{i(a)-1},
\pi ^{(n)}_{i(a)+1},\dots ,\pi ^{(n)}_N\}
$$
This means that 
in the case of a very special tower $F^0_{\d }$, the group 
$K^t_N(\c F)$ coincides with the limit of 
the projective system 
$(K_N^t(F^0_n))_{n\geqslant 0}$, where the connecting 
morphisms are the norm maps 
$N_{F^0_{n+1}/F^0_n}:K_N^t(F^0_{n+1})\longrightarrow K_N^t(F^0_n)$. 
We shall show  
that a similar property holds for any 
special SDR tower.

\begin{Prop} \label{P4.1} Suppose $K_{\d}$ is a special 
SDR tower with parameters $(n_0,c)$. Then for any 
$c_1\geqslant c$ and $n\geqslant n_0$, 
$$N_{K_{n+1}/K_n}U^{c_1}_K(K_N^t(K_{n+1}))
\subset U_K^{c_2}(K_N^t(K_n)),$$
where $c_2=c_1+c/p-e(K_{n+1}/K_0)^{-1}$. 
\end{Prop}

\begin{proof} Choose a field tower 
$K_n=L_0\subset L_1\subset \dots\subset L_N=K_{n+1}$, 
where for $1\leqslant i\leqslant N$, $[L_i:L_{i-1}]=p$. One can show 
the existence of:

--- a system of local parameters $\eta _1,\dots ,\eta _N$ of $L_0$;

--- a system of local parameters $\eta _1',\dots ,\eta _N'$ of $L_N$;

--- a permutation $\left (\begin{array}{cccc} 1&2&\dots &N\\
j_1&j_2&\dots &j_N\end{array}
\right )$
\newline 
such that for $1\leqslant i\leqslant N$, $L_i=L_{i-1}(\eta '_{j_i})$ and 
$\eta _i^{\prime p}\equiv\eta _i\,\mathrm{mod}\,\p _K^c$.

For any field extension $E$ of $K$ in $\bar K$, 
set 
$U_K^{c_1}(E):=(1+\p _K^{c_1})\cap E$. 

\begin{Lem} \label{L4.2} 
For all $1\leqslant i\leqslant N$, we have  
 
a) $N_{L_i/L_{i-1}}(\eta _{j_i}')\,\equiv\,\eta _{j_i}
\mathrm{mod}\,U_K^{c/p}(L_{i-1})$;

b) if $a\in L_i$ and $a-1\in\p _K^{c_1}$ then 
$N_{L_i/L_{i-1}}(a)-1\in\p _K^{c_1+c/p-v_K(\eta '_{j_i})}$. 
\end{Lem}

\begin{proof} The congruences 
$\eta _{j_i}^{\prime p}\equiv\eta _{j_i}\,\mathrm{mod}\,\p _K^c$ 
imply that all conjugates of $\eta _{j_i}'$ over $L_{i-1}$ are 
congruent modulo $\p _K^{c/p}$. Therefore, 
$N_{L_i/L_{i-1}}(\eta '_{j_i})\,\equiv\eta _i\,\mathrm{mod}\,\p _K^{c/p}$. 
This gives part a) of Lemma.

For property b), we can assume that $a=1+[\theta ]\eta ^{\prime b}_{j_i}
\in U_K^{c_1}(L_i)$, where $\theta\in K^{(N)}$ and $b\in \N $. Then 
all conjugates of $a$ over $L_{i-1}$ are congruent modulo 
${\eta _{j_i}'}^{b-1}\p _K^{c/p}=
{\eta _{j_i}'}^{-1}\p _K^{c_1+c/p}$ and this implies that 
$N_{L_i/L_{i-1}}(a)\equiv a^p\equiv 1\,
\mathrm{mod}\,({\eta _{j_i}'}^{-1}\p _K^{c_1+c/p})$.
\end{proof}

We continue the proof of Proposition \ref{P4.1}. 

Consider the following system of topological generators of the group 
$U_K^{c_1}(K_N^t(K_{n+1}))$. For any $a=(a_1,\dots ,a_N)\in\Z ^N_{>\bar 0}$, set 
$\varepsilon _{a\theta }=1+[\theta ]\underline{\eta }^a$, where $\theta\in K^{(N)}$ and 
$$\underline{\eta }^a=
{\eta '}^{a_1}_{j_1}\dots {\eta '}^{a_{s(a)-1}}_{j_{s(a)-1}}
\eta ^{p^{-1}a_{s(a)+1}}_{j_{s(a)+1}}\dots \eta _{j_N}^{p^{-1}a_N}.$$
Here $0\leqslant s(a)\leqslant N$ is such that 
$a_N\equiv\dots\equiv a_{s(a)+1}\equiv 0\,\mathrm{mod}\,p$ but 
$a_{s(a)}\not\equiv 0\,\mathrm{mod}\,p$. 
(Note that if $s(a)=0$, i.e. $a\in p\Z ^N$, then 
$\underline{\eta }^a\in K_n$.) 
With this notation, $U_K^{c_1}(K_{n+1})$ is topologically 
generated by all $\varepsilon _{a\theta }$ such that 
$v_K(\underline{\eta }^a)\geqslant c_1$. Therefore, 
$U_K^{c_1}(K_N^t(K_{n+1}))$ is topologically generated by the symbols 
$$\alpha _{a\theta i}=\{\varepsilon _{a\theta },
\eta '_{j_1},\dots ,\eta '_{j_{i-1}},\dots ,
\eta '_{j_{i+1}},\dots ,\eta '_{j_N}\}$$
with $a\in\Z ^N$, $a>\bar 0$, $1\leqslant i\leqslant N$ and 
$\varepsilon _{a\theta }\in U_K^{c_1}(K_{n+1})$. 
If $a\notin p\Z ^N$ then it will be enough 
to keep in this list only the generators with $i=s(a)$ and denote them by  
$\alpha _{a\theta }:=\alpha _{a\theta s(a)}$. 

By Lemma \ref{L4.2}a) and the properties of the norm map from Subsection 
\ref{S1.4}, we obtain that  
$N_{L_N/L_{s(a)}}(\alpha _{a\theta })$ is congruent to the symbol 
$
\{\varepsilon _{a\theta },\eta '_{j_1},\dots ,\eta '_{j_{s(a)-1}},
\eta _{j_{s(a)+1}},\dots ,\eta _{j_N}\}$ modulo 
$U^{c'}(K_N^t(L_{s(a)}))$ if $a\notin p\Z ^N$ 
with 
$$c'=c_1+c/p-\max\{v_K(\eta '_{j_{s(a)+1}}),\dots ,v_K(\eta '_{j_N})\}\leqslant 
c_1+c/p-e(K_{n+1}/K)^{-1}$$
Applying Lemma \ref{L4.2}b) and using that 
$\varepsilon _{a\theta }\in L_{s(a)}$ but 
$\eta '_{j_1},\dots ,\eta '_{j_{s(a)-1}}\in L_{s(a)-1}$, 
we obtain  that the norm 
$N_{L_N/L_{s(a)-1}}(\alpha _{a\theta })\in 
U_K^{c_2}(K_N^t(L_{s(a)-1}))$ for $c_2=c_1+c/p-e(K_{n+1}/K)^{-1}$. This implies that 
$N_{L_N/L_0}(\alpha _{a\theta })\in U_K^{c_2}(K_N^t(L_0))$. 

Similarly, we can consider the case $a\in p\Z ^N$, where the 
index $i$ plays the role of $s(a)$ and we use that 
$\varepsilon _{a\theta }\in L_0\subset L_{i-1}$. 
\end{proof}

\begin{Cor} \label{C4.3} For a special SDR tower $K_{\d }$, 
the limit of the projective system 
$(U_{K}^cK_N^t(K_n),N_{K_{n+1}/K_n})_{n\geqslant 0}$ 
is trivial.
\end{Cor} 

\begin{Prop} \label{P4.4} Suppose $K_{\d }$ is a special SDR tower with parameters 
$(n_0,c)$. Then there is  $0<c^0<c$ such that 
for all $n\gg 0$,  the morphisms $j_n$ 
in the projective system 
\eqref{E4.2} with $c$ replaced by $c^0$,
are induced by the norm maps 
$N_{K_{n+1}/K_n}:K_N^t(K_{n+1})\longrightarrow K_N^t(K_n)$.
\end{Prop}

\begin{proof} Let $n_0'\geqslant n_0$ be such that 
 $e(K_{n_0'+1}/K)<c/p$ and $c^0:=c/p-e(K_{n_0'+1}/K)$. 
Take $n\geqslant n_0'$ and use the notation from Proposition \ref{P4.1}. 
The group $K_N^t(K_{n+1})$ is topologically generated 
by the symbol $\{\eta _{j_1}',\dots ,\eta _{j_N}'\}$ and all 
$\alpha _{a\theta }=
\{\varepsilon '_{a\theta },\eta _{j_1},\dots ,\eta '_{j_N}\}$, where 
$a\in\Z ^N\setminus p\Z ^N$, $a>\bar 0$. Applying arguments from the proof of 
Proposition \ref{P4.1}, we obtain that 
$$N_{K_{n+1}/K_n}(\{\eta _{j_1}',\dots ,\eta _{j_N}'\})\equiv 
\{\eta ^{\prime p}_{j_1},\dots ,\eta ^{\prime p} _{j_N}\}\,\mathrm{mod}\,
U_K^{c^0}(K_N^t(K_n))$$ 
$$N_{K_{n+1}/K_n}(\alpha _{a\theta })\equiv 
\{\varepsilon ^p_{a\theta },
\eta ^{\prime p}_{j_1},\dots ,
\eta ^{\prime p}_{j_{s(a)-1}},
\eta ^{\prime p}_{j_{s(a)+1}},\dots ,
\eta ^{\prime p}_{j_N}\}
U_K^{c^0}(K_N^t(K_n)).$$
\end{proof}

\begin{Prop} \label{P4.5} Suppose $K_{\d }$ is a special SDR tower. 
Then for all $n\geqslant 0$, there are 
homomorphisms 
$\c N_{\c K/K_n}:K_N^t(\c K)\longrightarrow K_N^t(K_n)$ 
such that 

a) the system of morphisms $\{\c N_{\c K/K_n}\}_{n\geqslant 0}$ maps  $K_N^t(\c K)$ to the projective system 
$(K_N^t(K_n),N_{K_{n+1}/K_n})_{n\geqslant 0}$ and defines  
a group isomorphism 
$K_N^t(\c K)\overset{~}\simeq
\underset{n}\varprojlim K_N^t(K_n)$;

b) for sufficiently small $c>0$, the projective limit of the 
compositions of $\c N_{\c K/K_n}$ with projections 
$K_N^t(K_n)\rightarrow K_N^t(K_n)/U_K^c(K_N^t(K_n))$ 
coincides with the isomorphism $\iota ^{(c)}$ 
from Subsection \ref{S4.2}.
\end{Prop} 

\begin{proof} 
Suppose $(n_0',c^0)$ are the parameters for $K_{\d}$ from 
Proposition \ref{P4.4}. 
For any $a\in K_N^t(\c K)$ and $m\geqslant n_0'$, choose 
$a_m\in K_N^t(K_m)$ such that 
$\iota _m(a\,\mathrm{mod}\,U_{\c K}^{c^0}(K_N^t(\c K))=
a_m\,\mathrm{mod}\,U_K^{c^0}(K_N^t(K_m))$, where $\iota _m$ are 
isomorphisms rom Subsection \ref{S4.2}. 

This allows us to set  
$\c N_{\c K/K_n}(a)=\lim _{m\to\infty }N_{K_m/K_n}(a_m)$ 
for any $n\geqslant n_0'$. 

Then Corollary \ref{C4.3} can be used to prove that the maps 
$\c N_{\c K/K_n}$ are well-defined and satisfy the requirements of our Proposition.
 \end{proof}

Finally, notice the following properties:

\begin{Cor}\label{4.6} 
 Suppose $K_{\d }$ is a special SDR tower, 
$\c K=X(K_{\d })$ and $\bar u_1,\dots ,\bar u_N$ is a system of 
local parameters of $\c K$. Then 

a) $\underset{n\geqslant 0}\bigcap N_{K_n/K_0}K_N^t(K_n)=
\c N_{\c K/K_0}K_N^t(\c K)$;

b) for all $n\gg 0$, 
$\c N_{\c K/K_n}(\{\bar u_1,\dots ,\bar u_N\})=
\{u _1^{(n)},\dots ,u _N^{(n)}\}$,
where $u _1^{(n)},\dots ,u _N^{(n)}$ is a system of 
local parameters of $K_n$.
\end{Cor}

\subsection{The case of an arbitrary SDR tower $F_{\d}$}\label{S4.4}

Suppose $L/K$ is a finite Galois extension of $N$-dimensional 
local fields in a fixed algebraic closure 
$\bar K$ of $K$ and $G=\Gal (\bar K/K)$. Then there is a unique maximal purely 
 unramified extension $K_0$ of $K$ in $L$. 
This means that $[K_0:K]=[k_0:k]$, where $k_0$ and $k$ are the $N$-th 
residue fields of $L$ and $K$, respectively. Let $G_0=\Gal (L/K_0)$ and let 
$G_1=\{\tau\in G_0\ |\ \tau (a)/a\in 1+\m _L, \forall a\in L\}$, where 
$\m _L$ is the maximal ideal of the $N$-dimensional valuation ring 
$\c O_L$ of $L$. Then $G_1$ is a $p$-group, it is normal in $G$ and the order 
of $G_0/G_1$ is prime to $p$. The field extension $K_1:=L^{G_1}$ of $K$ is 
the maximal tamely ramified extension 
of $K$ in $L$. Note that in our setting,  
any tamely ramified field extension is always 
assumed to be Galois. Keeping the above notation we obtain the following property. 

\begin{Prop} \label{P4.7} 
Let $\widetilde{K}_1$ be a tamely ramified extension of $K$ 
 in $\bar K$ with the $N$-th residue field $k_0$ and $d=[\widetilde{K}_1:K_0]$. 
If $[\alpha ]$ is the Teichm\" uller 
representative of a generator $\alpha $ 
of $k^*_0$ and $u_1,\dots ,u_N$ is a system of local parameters in $K$ then 
$\widetilde{K}_1\subset K(\root d\of {[\alpha ]},
\root d\of{u_1},\dots ,\root d\of {u_N})$.
\end{Prop}

\begin{Cor} \label{C4.8} 
Suppose $\tilde d$ is a natural number and $\tilde d=p^md$, 
where $d$ is prime to $p$. 
 Let $\tilde k/k$ be a field extension of degree $\tilde d$. 
Choose a generator $\alpha $ of 
$\tilde k^*$, a system of local parameters $u_1,\dots ,u_N$
of $K$ and set 
$K(\tilde d):=K(\root d\of {[\alpha ]},\root d\of {u_1},
\dots ,\root d\of {u_N})$. Then $K(\tilde s)$  contains 
any tamely ramified extension of $K$ of degree dividing $\tilde d$.
\end{Cor}
 
\begin{Prop} \label{P4.9} 
Suppose $K_{\d }$ is an SDR tower with parameters $(n_0,c)$. 
Then there is a tamely ramified extension 
 $K_{n_0}'$ of $K_{n_0}$ such that the tower 
$K_{\d }'=\{K_nK_{n_0}'\ |\ n\geqslant 0\}$ is a special 
SDR tower.
\end{Prop}

\begin{proof} We need the following Lemma. 
 
\begin{Lem} \label{L4.10} 
Suppose $\widetilde{N}\in\N $ and 
$L/K$ is a totally ramified  separable extension 
of $K$ of degree $p^{\widetilde{N}}$, i.e. $L^{(N)}=k$. 
Let $\tilde d=(p^{\widetilde{N}})!=p^md$ where 
$m,d\in\N $ and $d$ is prime to $p$.  
Then there is a fields tower 
$L_0'\subset L_1'\subset\dots \subset L'_{\widetilde{N}}$ 
such that 

a) $L'_0$ is a tamely ramified extension of $K$ of degree dividing $\tilde d$;

b) for $1\leqslant i\leqslant N$, $[L'_{i}:L_{i-1}']=p$;

c) $L'_{\widetilde{N}}=L_0'L$.
\end{Lem}

\begin{proof}[Proof of Lemma] 
Since $L/K$ is separable 
there is $\theta\in L$ such that $L=K(\theta )$. Therefore, there is a Galois extension $\widetilde{K}$ of $K$ 
of degree dividing $\tilde d:=(p^{\widetilde{N}})!$ such that 
$L\subset \widetilde{K}$. Let $L_0'$ be the maximal tamely ramified 
extension of $K$ in $\widetilde{K}$. Then 
$L_0'\subset L'_{\widetilde{N}}:=L'_0L
\subset \widetilde{K}$, $[L'_{\widetilde{N}}:L'_0]=p^{\widetilde{N}}$ and 
$\widetilde{G}:=\Gal (\widetilde{K}/L'_0)$ is a finite $p$-group. Then 
elementary group theoretic arguments 
(e.g. any finite $p$-group has a central subgroup of order $p$) show the existence 
of a decreasing sequence of subgroups 
$$\widetilde{H}_0:=\widetilde{G}\supset \widetilde{H}_1\supset\dots \supset
\widetilde{H}_{\widetilde{N}-1}.\supset H_{\widetilde{N}}:=\Gal (\widetilde{K}/L'_{\widetilde{N}})$$
such that for $1\leqslant i\leqslant\widetilde{N}$, $(H_i:H_{i-1})=p$ and we can 
take $L_i'=\widetilde{K}^{H_i}$.
\end{proof}

For every $n\geqslant n_0$, let $K'_n:=L_0'$ where $L_0'$ is the field from 
the above Lemma chosen for $K=K_n$ and $L=K_{n+1}$. Since $[K'_n:K_n]$ 
divides $\tilde d=(p^{N})!$, 
the field $K'_n$ is contained in the tamely ramified extension $K_n(\tilde d)$ of $K_n$ from 
Corollary \ref{C4.8}. 
It remains to notice that for every $n\geqslant n_0$, $K_n(\tilde d)=K_{n_0}(\tilde d)K_n$. 
The Proposition is proved.
\end{proof}
 
For any $N$-dimensional local field $K$, set 
$K_N(K)_M:=K_N(K)/p^M$. If $L$ is a finite extension 
of $K$ and $c\geqslant 0$ denote by 
$U_K^cK_N(L)_M$ the image of $U_K^cK_N^t(L)$ in $K_N(L)_M$.

\begin{Prop} Suppose $F_{\d}$ is an SDR tower 
such that a primitive $p^M$-th root of unity 
$\zeta _M
\in F_{\infty }=
\underset{n\geqslant 0}\bigcup F_n$. Then 
for any $c>0$, $\underset{n}\varprojlim U_F^cK_N(F_n)_M=0$.
\end{Prop}

\begin{proof} We can assume that $F_{\d}$ has 
 $(0,c)$ and $\zeta _M\in F:=F_0$.
Let $F'$ be a tamely ramified extension of $F$ such that 
the SDR tower $F'_{\d }=\{F'F_n\ |\ n\geqslant 0\}$ is special. 
Let $C_n$ be the kernel of 
the natural map $U_F^cK_N(F_n)_M$ to $U_F^cK_N(F_n')_M$ 
induced by the embedding 
$F_n\subset F'_n$. By Corollary \ref{C4.3},  
$\underset{n}\varprojlim U^c_FK_N(F'_n)_M=0$. It remains 
to prove that $\underset{n}\varprojlim C_n=0$. 

Suppose $\widetilde{F}$ is the maximal absolutely unramified extension of $F$ 
in $F'$ of $p$-power degree. Then property (2) of 
norm maps from Subsection 
\ref{S1.4} implies that 
the natural map $K_N(F_n\widetilde{F})_M\longrightarrow 
K_N(F_nF')_M$ is injective. 
Therefore, for $n\geqslant 0$, $C_n$ is the kernel of the natural map 
from  $K_N(F_n)_M$ to 
$K_N(F_n\widetilde{F})_M$. 

Consider the Shafarevich bases of these $K_N$-groups from Subsection \ref{S1.3} 
using the Hasse $p^M$-primary 
elements $\epsilon _0$. By extending an $\F _p$-basis of $F_n^{(N)}$ to an $\F _p$-basis of 
$(F_n\widetilde{F})^{(N)}$, we can assume that 
the corresponding part of the Shafarevich basis for 
$K_N(F_n\widetilde{F})_M$ extends the corresponding part of 
the Shafarevich basis for $K_N(F_n)_M$. This implies that 
$C_n$ is contained in the subgroup 
of $K_N(F_n)_M$ which is generated by the elements of 
the Shafarevich basis which depend on $\epsilon _0$. 
Going again to the SDR tower $F_{\d }'$ we obtain that 
if $y$ is a such generator in $K_n(F_{n+1})_M$ 
then $N_{F_{n+1}/F_n}(y)\in pK_N(F_n)_M$. 
Therefore, $\underset{n}\varprojlim C_n=0$, as required.  
\end{proof}

\begin{Cor} \label{C4.12} 
Suppose $F_{\d}$ is an SDR tower such that a primitive 
 $p^M$-th root of unity $\zeta _M$ belongs to 
$F_{\infty }$ and $\c F=X(F_{\d })$. Then 

a) for $n\geqslant 0$, there is a system of homomorphisms 
$$\c N_{\c F/F_n}:K_N(\c F)_M\longrightarrow K_N(F_n)_M$$
mapping $K_N(\c F)_M$ to the projective system 
$(K_N(F_n)_M, N_{F_{n+1}/F_n})_{n\geqslant 0}$ and defining a group 
isomorphism $K_N(\c F)_M\simeq\underset{n}\varprojlim K_N(F_n)_M$;

b) for any $n\geqslant 0$, there is an $m\geqslant n$ such that 
$$\c N_{\c F/F_n}(K_N(\c F)_M)=N_{F_m/F_n}(K_N(F_m)_M).$$

c) for all $n\gg 0$, $\c N_{\c F/F_n}(\{\bar t_1,\dots ,\bar t_N\}_M)
=\{\pi _1^{(n)},\dots ,\pi _N^{(n)}\}_M$, 
where $\bar t_1,\dots ,\bar t_N$ and $\pi _1^{(n)},\dots ,\pi _N^{(n)}$ 
are the systems of local parameters in $\c F$ and, resp., $F_n$, and 
$\{\ ,\dots ,\ \}_M$ denotes the symbol 
$\{\ ,\dots ,\ \}$ taken modulo $p^M$.
\end{Cor}

\section{Applications to the Hilbert symbol}\label{S5}

As earlier let $F$ be an $N$-dimensional local field of characteristic 0 
with the first residue field of characteristic $p$ and a 
fixed system of local parameters $\pi _1$, \dots , $\pi _N$; we use the  
notation $\c F$ for an $N$-dimensional local field of characteristic $p$ 
with a system of local parameters $\bar t_1,\dots ,\bar t_N$. The last residue field of 
both $F$ and $\c F$ will be denoted by $k$.

\subsection{Parshin's reciprocity map.}\label{S5.1}

Consider the non-degenerate perfect Witt-Artin-Schreier pairing 
\begin{equation} \label{E5.1} 
W_M(\c  F)/\wp W_M(\c  F) \times 
\Gamma_{\c F}(p)/p^M \to \Z/p^M,
\end{equation} 
where for any $w\in W_M(\c F)$, $\wp (w)=\sigma (w)-w$, 
and $\Gamma _{\c F}(p)$ is the Galois group of the maximal 
abelian $p$-extension $\c F(p)$ of $\c F$. 
Then the Witt symbol, cf. Subsection \ref{S2.2}, 
implies the existence of injective group homomorphisms 
$$\widetilde{\Theta }^P_{\c F}:K_N^t(\c F)/p^M
\longrightarrow \Gamma _{\c F}(p)/p^M=
\Hom (W_M(\c F)/\wp (W_M(\c F)),\Z/p^M).$$
By taking the projective limit over  
$M\in\N $, we obtain an injective homomorphism 
$\widetilde{\Theta }^P_{\c F}:K_N^t(\c F)
\longrightarrow \Gamma _{\c F}(p)$. 
Here, as in the Introduction, we use 
the same notation 
for homomorphisms and their reductions modulo $p^M$. 
Then the explicit formula 
for the Witt symbol from  
Subsection \ref{S2.2}   
implies that 
$\widetilde{\Theta }^P_{\c F}$ is 
$P$-continuous. In \cite{Pa1,Pa2,Pa3} Parshin used 
$\widetilde{\Theta }^P_{\c F}$ to develop class 
field theory of higher local fields of 
characteristic $p$ by proving that for all 
finite abelian extensions $\c E$ of $\c F$ in 
$\c F(p)$, $\widetilde{\Theta }^P_{\c F}$ 
induces the group isomorphisms 
$$\widetilde{\Theta }^P_{\c E/\c F}:
K_N^t(\c F)/N_{\c E/\c F}K_N^t(\c E)\longrightarrow 
\Gal (\c E/\c F).$$
The explicit formula for the Witt symbol implies 
that the intersection of all $N_{\c E/\c F}K_N^t(\c E)$, where 
$\c E$ runs over the family of all finite abelian 
extensions of $\c F$, is trivial. Therefore, 
$\widetilde\Theta ^P_{\c F}$ is the composition of 
the canonical embedding 
$$j_{\c F}:K_N^t(\c F)\longrightarrow \widehat{K}_N^t(\c F)
:=\underset{\c E}\varprojlim K_N^t(\c F)/N_{\c E/\c F}K_N^t(\c E)$$
and the isomorphism 
$
\widehat{\Theta}^P_{\c F}:=\underset{\c E}\varprojlim 
\widetilde{\Theta }^P_{\c E/\c F}:
\widehat{K}_N^t(\c F)\longrightarrow \Gamma _{\c F}(p)$.

This gives the morphisms 
\begin{equation}\label{E5.1a}
\Gamma _{\c F}^{ab}/p^M\overset{\Theta ^P_{\c F}}\longrightarrow 
\widehat{K}_N^t(\c F)/p^M\overset{j_{\c F}}\longleftarrow K_N^t(\c F)/p^M,
\end{equation}
where $\Theta ^P_{\c F}=
\left (\widetilde{\Theta }_{\c E/\c F}^{P}\right )^{-1}$ is 
a group isomorphism  
and $j_{\c F}$ is embedding with a dense image.

Let $\c F^{ur}$ be the maximal absolutely 
unramified extension of $\c F$ in $\c F(p)$. Denote by 
$\varphi _{\c F}\in\Gal (\c F^{ur}/\c F)$ 
the Frobenius automorphism of this 
extension. Let $\c F^+$ be the subfield in $\c F(p)$ 
invariant under the action of 
$\widetilde{\Theta }^P_{\c F}(\{\bar t_1,\dots ,\bar t_N\})$. 
Using the explicit formula for the pairing \eqref{E2.2a} from 
Subsection \ref{S2.2} one can easily obtain the following properties:

a) $\c F(p)=\c F^+\c F^{ur}$ and $\c F^+\cap\c F^{ur}=\c F$;

b)\ $\{\bar t_1,\dots ,\bar t_N\}$ generates the group 
$\underset{\c E}\bigcap N_{\c E/\c F}K_N^t(\c E)$, where 
$\c E$ runs over the family of all 
$\c F\subset \c E\subset \c F^+$ such that 
$[\c E:\c F]<\infty $;

c)\ $\widetilde{\Theta }_{\c F}^P(VK_N(\c F))$ is  a closed subgroup in 
$\Gamma _{\c F}(p)^{ab}$ and its invariant subfield is $\c F^{ur}$;

d) $\widetilde{\Theta }_{\c F}^P
(\{\bar t_1,\dots ,\bar t_N\})|_{\c F^{ur}}=\varphi _{\c F}$;

e) $VK_N^t(\c F)=\underset{\c E}\bigcap 
N_{\c E/\c F}K_N^t(\c E)$, where $\c E$ runs over 
the family of all finite extensions of $\c F$ in $\c F ^{ur}$.
\medskip 

%%%%%%%%%%%%%%%%%%%%%%%%%%%%%%%%%%
%%%%%%%%%%%%%%%%%%%%%%%%%%%%%%%%%%%%%
%%%%%%%%%%%%%%%%%%%%%%%%%%%%%%%%%%%%
%%%%%%%%%%%%%%%%%%%%%%%%%%%%%%%%%%

\subsection{Fesenko's reciprocity map.} \label{S5.2} 

 In \cite{Fe}, Fesenko defined the reciprocity 
map for higher local fields of arbitrary characteristic. This construction 
can be specified in our situation as follows.

Let $K$ be either $\c F$ or $F$. Let $L$ 
be a finite extension of $K$ in its maximal abelian $p$-extension 
$K(p)$. Denote by 
$L^{ur}$ and $K^{ur}$ the maximal absolutely unramified 
extensions of $L$ and, resp. $K$, in $K(p)$. 
Set $L_0=K^{ur}\cap L$. 

For any $\tau\in\Gal (L/K)$, consider its lift 
$\tilde\tau \in\Gal (L^{ur}/K)$ such that 
$\tilde\tau |_{K^{ur}}=\varphi _K^i$, where 
$i\in\N$ and $\varphi _K$ is the Frobenius 
automorphism of $K^{ur}/K$. Such lift exists because 
$K^{ur}$ and $L$ are linearly disjoint over $L_0$. 
Let $\Sigma $ be the subfield 
of $\tilde\tau $-invariants in $L^{ur}$. 
Then $\Sigma $ is a finite extension of $K$ 
such that $\Sigma K^{ur}=L^{ur}$. Denote by 
$\pi _1^{\Sigma },\dots ,\pi _N^{\Sigma }$ 
a system of local parameters in $\Sigma $. 
Then Fesenko's reciprocity map is defined to be  
$$\Theta ^{\Phi }_{L/K}:\Gal (L/K)\longrightarrow 
K^t_N(K)/N_{L/K}(K_N^t(L))$$
such that for any $\tau\in \Gal (L/K)$, 
$\Theta ^{\Phi }_{L/K}(\tau )=
\cl _{L/K}(N_{\Sigma /K}(\{\pi _1^{\Sigma },\dots ,
\pi _N^{\Sigma }\})$, 
 where for any $\alpha\in K_N^t(K)$, 
$\cl _{L/K}(\alpha )$ is the image of $\alpha $ 
under the natural projection of 
$K_N^t(K)$ to $K_N^t(K)/ 
N_{L/K}(K_N^t(L))$. The maps $\Theta ^{\Phi }_{L/K}$ 
are well-defined group homomorphisms.  
The proof generalises  
Neukirch's 1-dimensional approach from \cite{Ne}.

\begin{remark} Fesenko establishes his construction of class field 
theory for higher local fields by proving that 
all $\Theta ^{\Phi }_{L/K}$ are isomorphisms. We do not  
assume this until the introduction of the $M$-th 
Hilbert symbol in Subsection \ref{S5.5}. 
\end{remark}

Taking projective limit we obtain Fesenko's reciprocity map in the form 
$$\Theta ^{\Phi }_K:\Gamma _K^{ab}
\longrightarrow\underset{L}\varprojlim K_N^t(K)/
N_{L/K}K_N^t(L):=\widehat{K}_N^t(K).$$

The following properties follow directly from 
the above definitions.

a) Suppose $L\subset K^{ur}$ and 
$\tau =\varphi _K|_L$. Then 
$\Theta ^{\Phi }_{L/\c K}(\tau )=
\cl _{L/\c K}(\{u_1,\dots ,u_N\})$, 
where $u_1,\dots ,u_N$ is a system of local parameters in $K$. 
\medskip 

b) Suppose $L/K$ is a finite abelian extension, 
$K_1$ is a finite field extension of $K$ and $L_1=LK_1$. 
Then there is a natural group homomorphism 
$\kappa :\Gal (L_1/K_1)\longrightarrow \Gal (L/K)$ and 
the diagram 
$$\xymatrix{\Gal (L_1/K_1)
\ar[d]^{\kappa}\ar[rrr]^{\Theta ^{\Phi }_{L_1/K_1}}&&&
K_N^t(K_1)/N_{L_1/K_1}K_N^t(L_1)\ar[d]^{N_{K_1/K}}\\
\Gal (L/K)\ar[rrr]^{\Theta ^{\Phi }_{L/K}}
&&& K_N^t(K)/N_{L/K}K_N^t(L)
}$$
is commutative. 

The following proposition shows that Parshin's and Fesenko's reciprocity 
maps coincide in the case of fields of characteristic $p$.

\begin{Prop} \label{P3.1}
 Suppose $\c E$ is a finite abelian extension of 
$\c F$ in $\c F(p)$. Then the diagram 
$$\xymatrix{\Gamma _{\c F}(p)
\ar[d]&&&
K_N^t(\c F)\ar[lll]_{\widetilde{\Theta} ^P_{\c F}}\ar[d]\\
\Gamma _{\c E/\c F}\ar[rrr]^
{\Theta ^{\Phi }_{\c E/\c F}}
&&&K_N^t(\c F)/N_{\c E/\c F}K_N^t(\c E)
}$$
is commutative, where the vertical 
maps are the natural projections.  
\end{Prop} 

\begin{proof} The group $K^t_N(\c F)$ is generated by 
the symbols $\{\bar t_1',\dots ,\bar t_N'\}$, where 
$\bar t'_1,\dots ,\bar t'_N$ run over all systems of local parameters in 
$\c F$. 
Therefore, it will be sufficient to consider the 
images of $\alpha =\{\bar t_1,\dots ,\bar t_N\}$. 
By Subsection \ref{S3.1} 
we have $\widetilde{\Theta}_{\c F}^P
(\alpha )|_{\c F^{ur}}=\varphi _{\c F}$ 
and $\widetilde{\Theta }_{\c F}^P(\alpha )|_{\c F^+}=\id $. 
According to properties a) and b) of Subsection \ref{S5.2}, this implies that 
$\Theta ^{\Phi}_{\c E/\c F}
(\widetilde{\Theta} ^P(\alpha ))$ 
belongs to 

--- 
$\cl _{\c E/\c F}
\left (\underset{\c E_1}\cap N_{\c E_1/\c F}
(K_N^t(E_1))\right )$ , where $\c E_1$ runs over 
the set of all finite extensions of $\c F$ in $\c F^+$, 
which is a group generated by 
$\cl _{\c E/\c F}(\alpha )$;

--- $\cl _{\c E/\c F}(\alpha\,
\mathrm{mod}\,VK_N(\c F))$.

Therefore, $\Theta ^{\Phi}_{\c E/\c F}
(\widetilde{\Theta} ^P(\alpha ))
=\cl _{\c E/\c F}(\alpha )$.
 \end{proof}

%%%%%%%%%%%%%%%%%%%%%%%

%%%%%%%%%%%%%%%%%%%%%%%%%%%%%%%%%%%%%%%%%%%%%%%%%%%%%%%%%%%%%%%%%%%%%%%%%%%%%%%%%%%
\subsection{Compatibility of class-field theories} \label{C5.3}

Suppose $E$ is a finite field extension of 
$F$ and 
$F_{\d }=(F_n)_{n\geqslant 0}$ with $F_0=F$, 
is a special SDR tower. Then $E_{\d }=(E_n)_{n\geqslant 0}$, where 
$E_n=EF_n$, is also a special SDR tower and 
$\c E=X(E_{\d })$ is a finite separable extension of 
$\c F=X(F_{\d })$. Notice that for any $n\geqslant 0$, there is a commutative diagram 

$$\xymatrix{K_N^t(\c E)\ar[rrr]^{N_{\c E/\c F}}
\ar[d]^{\c N_{\c E/E_n}}&&&
K_N^t(\c F)\ar[d]^{\c N_{\c F/F_n}}\\
K_N^t(E_n)\ar[rrr]^{N_{E_n/F_n}}&&&
K_N^t(F_n)
}$$

We shall use the notation $\c N_{\c E/F, n}$ for the  
 morphism 
$$K_N^t(\c F)/N_{\c E/\c F}(K_N^t(\c E)
\longrightarrow K_N^t(F_n)/N_{E_n/F_n}K_N^t(E_n)$$
induced by $\c N_{\c F/F_n}$. 
 
Suppose that $E$ is abelian over $F$. Then $\c E$ is also abelian over $\c F$, 
and for any 
$n\geqslant 0$, we have natural homomophisms  
$$\iota _{\c E/F,n}:\Gal (\c E/\c F)\longrightarrow 
\Gal (E_n/F_n)$$
which are isomorphisms for $n\gg 0$. 

Let $\iota _{\c E/\c F,0}:=\iota _{\c E/\c F}$ and 
$\c N_{\c E/\c F,0}:=\c N_{\c E/\c F}$.

\begin{Prop} \label{P3.9} The diagram 
 $$\xymatrix{\Gal (\c E/\c F)
\ar[rrr]^{\Theta ^{\Phi }_{\c E/\c F}}
\ar[d]^{\iota _{\c E/F}}&&&
K_N^t(\c F)/N_{\c E/\c F}K_N^t(\c E)
\ar[d]^{\c N_{\c E/\c F}}\\
\Gal (E/F)\ar[rrr]^{\Theta ^{\Phi }_{E/F}}
&&& K_N^t(F)/N_{E/F}K_N^t(E)
}$$
is commutative. 
\end{Prop}

\begin{proof} Let $\tau\in\Gal (\c E/\c F)$. Construct 
 $\tilde\tau\in\Gal (\c E^{ur}/\c F)$ and 
$\c S=\c E^{ur}|_{\tilde\tau =\id }$ to define 
Fesenko's element 
$$\Theta ^{\Phi }_{\c E/\c F}(\tau )=
N_{\c S/\c F}(\{\bar u_1,\dots ,\bar u_N\})
\in K_N^t(\c K)/N_{\c L/\c K}K_N^t(\c L).$$
where $\bar u_1,\dots ,\bar u_N$ is a system of 
local parameters of $\c S$. 

For any $n\geqslant 0$, consider an analogue 
$$\iota _{\c E^{ur}/\c F,n}:\Gal (\c E^{ur}/\c F)
\longrightarrow\Gal (E_n^{ur}/F_n)$$
of $\iota _{\c E/\c K,n}$. 
Let $\tau _n=\iota _{\c E/\c F,n}(\tau )$, 
$\tilde\tau _n=\iota _{\c E^{ur}/\c F,_n}(\tilde\tau )$ and 
set $\Sigma _n=L_n^{ur}|_{\tilde\tau _n=\id }$. 
Then from the construction of the field-of-norms functor $X$ 
it follows that $\Sigma _{\d }=(\Sigma _n)_{n\geqslant 0}$ 
is an SDR tower and $X(\Sigma _{\d })=\c S$. Therefore, 
for $n\gg 0$, 
$$\c N_{\c E/\c F,n}(\Theta ^{\Phi }_{\c E/\c F}(\tau ))
=\c N_{\c E/\c F,n}(N_{\c S/\c F}
(\{\bar u_1,\dots ,\bar u_N\})\,
\mathrm{mod}\,N_{\c E/\c F}K_N^t(\c E))$$
$$=N_{\Sigma _n/F_n}(\c N_{\c S/\Sigma _n}(\{\bar u_1,\dots ,\bar u_N\})
\,\mathrm{mod}\,N_{E_n/F_n}K_N^t(E_n))=
\Theta ^{\Phi }_{E_n/F_n}(\tau _n).$$
It remains to apply the property b) from Subsection \ref{S5.2}.
\end{proof}

Finally, we can use the results of Subsection \ref{S4.4} 
to establish the compatibility of class field 
theories for the fields $F$ and $\c F=X(F_{\d })$ 
if $F_{\d }$ is an arbitrary SDR tower such that $\zeta _M\in F_{\infty }$.  
This property can be stated in the following form. 

\begin{Cor} \label{C3.10} 
With the above notation and assumptions  
one has the following commutative diagram
$$\xymatrix{
\Gamma _{\c F}^{ab}/p^M\ar[rrr]^{\Theta ^\Phi _{\c F}}
\ar[d]^{\iota _{\c F/F}}&&&
\widehat{K}_N(\c F)/p^M
\ar[d]_{\widehat{\c N}_{\c F/F}} && K_N(\c F)/p^M \ar[ll]
\ar[d]^{\c N_{\c F/F}}\\
\Gamma _F^{ab}/p^M \ar[rrr]^{\Theta ^\Phi _F} &&&
\widehat{K}_N(F)/p^M && K_N(F)/p^M \ar[ll]
}$$
where the right horizontal maps are natural embeddings and 
$\iota _{\c F/F}:\Gamma _{\c F}\longrightarrow\Gamma _F$ is 
given by the field-of-norms functor. 
\end{Cor}
\medskip 

\subsection{Relating Witt-Artin-Schreier and Kummer theories}
\label{S5.4}

Consider an $N$-dimensional analogue $R(N)$ of Fontaine's ring. 
By definition, $R(N)=\underset{n}\varprojlim (O_{\bar F}/p)_n$, where 
the connecting morphisms are induced by the $p$-th power map on $O_{\bar F}$. 
If $r=(r_n\,\mathrm{mod}\, p)_{n\geqslant 0}\in R(N)$ with 
all $r_n\in \c O_{\bar F}$ and $m\in\Z $, set 
$r^{(m)}:=\lim _{n\to\infty }r_n^{p^{n+m}}\in\c O_{\hat F_{\infty }}$ 
and consider Fontaine's map 
$\gamma :W(R(N))\longrightarrow \c O_{\hat F_{\infty }}$ given by the 
correspondence 
$$(w_0,\dots ,w_n,\dots )\mapsto\sum _{n\geqslant 0}p^nw_n^{(0)}.$$

Let $F_{\d }$ be an SDR tower with parameters $(0,c)$. 
Then we have natural embeddings 
$$\c O_{\c F}=\underset{n}\varprojlim \,\c O_{F_n}/
\p _F^c\subset\underset{n}\varprojlim \,\c O_{\bar F}/\p _F^c=
\underset{n}\varprojlim \,\c O_{\bar F}/p=R(N),$$
where $\c O_{\c F}, \c O_{F_n}$ 
and $\c O_{\bar F}$ are the corresponding 
$N$-dimensional valuation rings. 
This implies that $\c F\subset R_0(N):=\mathrm{Frac}\, R(N)$.
Note that $R_0(N)$ is algebraicly closed and equals the completion 
(with respect to the first valuation) 
of the algebraic closure  of $\c F$ in $R_0(N)$. We have also 
a natural embedding $\c O_{L(\c F)}\subset W(R_0(N))$. 
In particular, $\Gamma _{\c F}$ acts on $R_0(N)$. 
In terms of the fixed system of local parameters in 
$L(\c F)$, $\c F$ and $F_n$, where $n\geqslant 0$, we have for all 
$1\leqslant i\leqslant N$, that 
$\bar t_i=\underset{n}\varprojlim \pi _i^{(n)}$, 
$t_i=[\bar t_i]\in W(R(N))$ and $\gamma (t_i)=\underset{n\to\infty }
\lim\pi _i^{(n)p^n}$. 

Suppose $\omega\in\Z _{\geqslant 0}$ and $F_{\d }$ is $\omega $-admissible, 
cf. Introduction for the deinition of an $\omega $-admissible tower. 
Then we can fix a primitive $p^{M+\omega }$-th primitive root of unity 
$\zeta _{M+\omega }\in F_{\omega }$ and introduce an 
element $H_{\omega }\in\c O_{L(\c F)}$ as follows.

Let $H'\in\c F$ be such that 
\begin{equation}\label{E5.4} 
H'\,\mathrm{mod}\,\p _{\c F}^{p^\omega c}=
 \zeta _{M+\omega }\,\mathrm{mod}\,\p _F^c
\end{equation}
under the identification 
$\c O_{\c F}/\p _{\c F}^{p^\omega c}=
\c O_{F_{\omega }}/\p _F^c$ from 
the definition of $\c F=X(F_{\d })$. 
Take any $H\in\c O_{L(\c F)}$ such that 
$H\,\mathrm{mod}\,p=H'$ and set 
$H_{\omega }=H^{p^{M+\omega }}-1$. 

Suppose $f\in\m ^0=
\{\underset{a>\bar 0}\sum w_a\underline{t}^a\ |\ w_a\in W(k)\}$. 
For any $\tau\in\Gamma _{\c F}$, let $a_{\tau }(f)=\tau (T)-T$, where 
$T\in W(R_0(N))$ is such that $\sigma (T)-T=f/H_{\omega }$. Clearly, 
$a_\tau (f)\in\Z _p$ for all $\tau\in\Gamma _{\c F}$. 

Suppose $g\in\hat F_{\infty }^*$. For any $\tau\in\Gamma _{\c F}$, 
define $b_{\tau }(g)\in\Z /p^M$ such that $(\tau U)U^{-1}=
\zeta _{M+\omega }^{p^{\omega }b_{\tau }(g)}$, where 
$U\in\hat{\bar F}$ is such that $U^{p^M}=g$. 
We use the identification $\Gamma _{\c F}=\Gal (\bar F/F_{\infty })$ given by the 
field-of-norms functor $X$. 

Recall that 
$$\underset{a>\bar 0}\sum w_a\underline{t}^a
\mapsto \gamma \left (\underset{a>\bar 0}\prod E(w_a,\underline{t}^a)\right )$$
defines a homomorphism 
$\theta :\m ^0\longrightarrow \hat F_{\infty }^*$, 
where $E(w,X)$ is the Shafarevich generalisation of the Artin-Hasse exponential, 
cf. Introduction.

\begin{Prop} \label{P5.4} 
For any $f\in\m ^0$ and $\tau\in\Gamma _{\c F}$, 
we have the equality 
$a_{\tau }(f)\,\mathrm{mod}\,p^M=b_\tau (\theta (f))$. 
\end{Prop}

\begin{proof} Let $\varepsilon\in R(N)$ be such that 
$\varepsilon ^{(0)}=1$ but $\varepsilon ^{(1)}\ne 1$. 
We assume that $\varepsilon ^{(M+\omega )}=\zeta _{M+\omega }$. 
Then relation \eqref{E5.4} implies that 
$$H\,\mathrm{mod}\,p= \varepsilon ^{p^{-(M+\omega )}}
\,\mathrm{mod}\,\p _{\c F}^{p^{\omega }c}.$$
The ideal $\p _{\c F}^{p^{\omega }c}$ is generated by 
the element $(\varepsilon -1)^{c(p-1)/pv_F(p)}$ and, therefore, 
$$H^{p^{\omega +1}}\,\mathrm{mod}\,p=\varepsilon ^{p^{-(M-1)}}
\,\mathrm{mod}\,(\varepsilon -1)^{cp^{\omega }(p-1)/v_F(p)}.$$
Using that $F_{\d}$ is $\omega $-admissible, we obtain 
the existence of $w_1\in W(R(N))$ and $w_1^0\in W(R_0(N))$ such that 
$$H^{p^{\omega +1}}=[\varepsilon ]^{p^{-(M-1)}}+
([\varepsilon ]-1)^2w_1+pw_1^0.$$
Therefore, there are $w_2\in W(R(N))$ and $w_2^0\in W(R_0(N))$ such that 
$$H_{\omega }=[\varepsilon ]-1+([\varepsilon ]-1)^2w_2+p^Mw_2^0$$
and for some $w\in W(R(N))$ we have 
$$\frac{1}{H_{\omega }}\equiv 
\left (\frac{1}{[\varepsilon ]-1}+w\right )\,\mathrm{mod}\,p^MW(R_0(N)).$$

For any $\tau\in\Gamma _{\c F}$, let $a'_{\tau }(f)=\tau T'-T'\in\Z _p$, where 
$T'\in W(R_0(N))$ is such that $\sigma (T')-T'=f/([\varepsilon ]-1)$. 
Clearly, $\underset{s\to\infty }\lim \sigma ^s(fw)=0$ and this implies that 
$a'_{\tau }(f)\equiv a_{\tau }(f)\,\mathrm{mod}\,p^M$. 

Now one can proceed along the lines of the Main Lemma from \cite{Ab1} (or cf also 
\cite{Ab2}) to establish that 
$a'_{\tau }(f)\,\mathrm{mod}\,p^M=b_{\tau }(\theta (f))$.
\end{proof} 
\medskip 

\begin{Cor} \label{C5.5} 
 If $\alpha _0\in W(k)$ is such that $\Tr (\alpha _0)=1$ 
then  \newline 
{\rm a)} $\theta (\alpha _0H_0)$ is a $p^M$-primary element of $F$;
\newline 
{\rm b)} if $\phi $ is the Frobenius automorphism of the extension 
$F_{ur}/F$ then $\phi (\theta (\alpha _0H_0))=\zeta _M(\theta (\alpha _0H_0))$. 
\end{Cor}

\begin{proof} Indeed, $\theta (\alpha _0H_0)\in 1+\m _F$ and therefore  
 we can study $\Gamma _F$-properties of the extension 
$F(\root {p^M}\of {\theta (\alpha _0H_0)})$ by studying 
$\Gamma _{\c F}$-properties of the extension $\c F(T)$, where 
$T\in W_M(R_0(N))$ is such that $\sigma (T)-T=\alpha _0$. But 
this extension is absolutely unramified of degree $p^M$ 
by Witt's explicit formula from Section \ref{S2}). 
This proves part a). In order to prove b), it is sufficient to 
note that $\phi (T)-T=\sigma ^s(T)-T=\Tr (\alpha _0)=1$, where 
$[k:\F _p]=s$ and then to apply Proposition \ref{P5.4}. 
\end{proof}

\subsection{Proof of Theorem \ref{T0.2}} \label{S5.5}

Suppose $F_{\d }$ is an SDR $\omega $-admissible tower with 
parameters $(0,c)$, $F_0=F$, $\c F=X(F_{\d })$ and  
$\beta \in K_N(\c F)$. Then there is an $\tau\in\Gamma _F^{ab}$ such that 
(in the notation of Subsection \ref{S4}) 
$$\Theta _F^{\Phi }(\tau\,\mathrm{mod}\, p^M)=\c N_{\c F
/F}(\beta \,\mathrm{mod}\,p^M).$$ 
By Corollary \ref{C3.10}, there is an 
$\tilde\tau\in\Gamma _{\c F}^{ab}$ such that 
$\tau =\iota _{\c F/F}(\tilde\tau )$ and $\Theta _F^{\Phi }
(\tilde\tau )=\beta $, using the notation from the Introduction. 

For any $f\in\m ^0$, let 
$(\Theta (f), \c N_{\c F/F}(\beta ))^{F_{\d }}_M=
\zeta _{M+\omega }^{p^\omega A}$, where $A\in\Z /p^M$. 

We construct the corresponding $H_{\omega }\in\m ^0$, cf. Introduction 
and use Proposition \ref{P5.4} to deduce that:

--- if $U\in W(R_0(N))$ is such that $\sigma (U)-U=f/H_{\omega }$ 
then $$\tilde\tau U-U)\,\mathrm{mod}\,p^M=A.$$

Finally, by Proposition \ref{P2.5},
$$A=\Tr \left (\Res _{L(\c F)}
\frac{f}{H_{\omega }}d_{\log }(\mathrm{Col} \beta )\right ).$$
Theorem \ref {T0.2} is proved. 
\medskip 

\subsection{Relation to Vostokov's pairing}  

Suppose that $\zeta _M\in F$ and $F_{\d }^0=\{F_n^0\ |\  n\geqslant 0\}$
is a very special tower given in notation of Subsection \ref{S4.3} 
such that $F_0^0=F$. Recall that each $F_n^0$ has a system 
of local parameters $\pi _1^{(n)},\dots ,\pi _N^{(n)}$ 
such that for $1\leqslant i\leqslant N$, $\pi _i^{(0)}=\pi _i$ 
and $\pi _i^{(n+1)p}=\pi _i^{(n)}$. Then  
$F_{\d }^0$ is a $0$-admissible SDR tower. As earlier,  
$\c F=X(F_{\d }^0)$ with system of local parameters 
$\bar t_i=\underset{n}\varprojlim \pi _i^{(n)}$, where 
$1\leqslant i\leqslant N$, and $L(\c F)$ is the corresponding 
absolutely unramified lift of $\c F$ to characteristic 0 
with local parameters $p,t_1,\dots ,t_N$ such that 
$t_i\,\mathrm{mod}\,p=\bar t_i$, $1\leqslant i\leqslant N$. 
We have the following result. 

\begin{Thm} \label{Th1} For above SDR tower $F_{\d }^0$, the explicit  
formula for the $M$-th Hilbert symbol from 
Theorem \ref{T0.2} coincides with Vostokov's pairing. In other words, for 
very special towers the field-of-norms functor transforms  
Witt's pairing to Vostokov's pairing. 
\end{Thm}

\begin{proof} Note that the very special tower $F_{\d }^0$ 
 has the following advantages:

--- the map $\gamma :1+\m ^0\longrightarrow F$ is given by the 
correspondences $t_i\mapsto \pi _i$, $1\leqslant i\leqslant N$, and hence  
coincides with the evaluation map $\kappa $ from the beginning 
of section \ref{S3};

--- the Coleman map $\mathrm{Col} :K_N^t(\c F)\longrightarrow K_N^t(L(\c F))$ 
has a very simple explicit description in terms of the standard topological 
generators of the corresponding $K$-groups, cf. the beginning 
of Subsection \ref{S2.3}. 

It will be sufficient to verify the coincidence of the both explicit formulae  
on the standard topological generators of $F^*/p^M$ and $K_N(F)/p^M$ 
from Subsections \ref{S1.2} and \ref{S1.3}. It can be seen  
that on these generators (due to the above mentioned properties of 
very special towers) the formula from Theorem \ref{T0.2} 
coincides with the \lq\lq $i=0$\rq\rq -term of Vostokov's formula. 
In the notation of Section \ref{S3} we, therefore, need to verify that 
for $1\leqslant i\leqslant N$, the 
$i$-parts 
$$V_i:=\Tr (\Res H_0^{-1}f_i(\sigma /p)d_{\log}u_0\wedge\dots 
\wedge (\sigma /p)d_{\log}u_{i-1}\wedge d_{\log}u_{i+1}\wedge
\dots\wedge d_{\log}u_N)$$
of Vostokov's formula give a zero contribution on standard generators. 

The variable $u_0$ can take the following values:

$a_1$) $t_j$, where $1\leqslant j\leqslant N$;

$a_2$) $1+[\theta ]\underline{t}^a$, where $\theta\in k$ and 
$a=(a_1,\dots ,a_N)\in\Z ^N\setminus p\Z ^N$, $a>\bar 0$;

$a_3$) $\alpha _0H_0$ in the notation from Subsection \ref{S5.5}.

The symbol $\{u_1,\dots ,u_N\}$ can take the following values 
(the generators containing $\epsilon _0$ do not come from 
$K_N(\c F)$):

$b_1$) $\{t_1,\dots ,t_N\}$;

$b_2$) $\{1+[\theta ']\underline{t}^b,t_1,\dots ,t_{i(b)-1},t_{i(b)+1},
\dots ,t_N\}$, 
where $\theta '\in k$, 
$b=(b_1,\dots ,b_N)$ belongs to $\Z ^N\setminus p\Z ^N$, $b>\bar 0$, 
$b_N\equiv\dots b_{i(b)+1}\equiv 0\,\mathrm{mod}\,p$ and 
$b_{i(b)}\not\equiv 0\,\mathrm{mod}\,p$.

In the case $b_1$) $V_i=0$ for any $u_0$, because $f_i=0$. 

In the case $a_3$)  $V_i=0$ for any $\{u_1,\dots ,u_N\}$, because 
$(\sigma /p)d_{\log}H_0\in p^M\Omega _{O^0}$.

In the case $a_1b_2$) we can assume that $j=i(b)$ and $i=1$. 
Then the differential form from the expression of $V_1$ is a linear 
combination of the differential forms 
$H_0^{-1}\underline{t}^{ub} d_{\log}t_1\wedge\dots\wedge d_{\log}t_N$  
for $u\in\N$, $u\not\equiv 0\,\mathrm{mod}\,p$. All these 
differential forms have zero residue because 
$H_0\in\sigma(\m ^0)\mathrm{mod}p^M$.

Similarly, in the case $a_2b_2$) we can assume that $i=1$ and that the 
corresponding differential form is a linear combination of the forms 
$H_0^{-1}\underline{t}^{spa+ub}d_{\log}t_1\wedge\dots\wedge d_{\log}t_N$  
for $s,u\in\Z_{\geqslant 0}$ and $u\not\equiv 0\mathrm{mod}\,p$. 
All these forms have zero residue for the same reason. 

The Theorem is proved.
\end{proof}


\begin{thebibliography}{xxx}


\bibitem{Ab1} {\sc V.\,Abrashkin} \textit {The field of norms 
functor and the Brueckner-Vostokov formula}, 
{Math. Annalen} {\bf308} (1997), 5-19 

\bibitem{Ab2} {\sc V.\,Abrashkin} \textit{An analogue of the field-of-norms functor 
and of the Grothendieck Conjecture}, {J.\,Algebraic Geom.} {\bf 16} (2007), no.\, 4, 671-730

\bibitem{BV} {\sc T.\,V.\,Belyaeva, S.\,V.\, Vostokov} 
\textit{The Hilbert symbol in a complete multidimensional field.I}, 
{J.Math.Sci (N.Y.)} {\bf 120} (2004), no.\, 4, 1483-1500

\bibitem{BK} {\sc S.\,Bloch, K.\,Kato} \textit{$p$-adic \'etale cohomology},  
{Publ.\, Math.\, IHES} {\bf 63} (1986), 107-152

\bibitem{BT} {\sc Bass, J.\,Tate}, \textit{The Milnor ring of a global field}, 
Lect. Notes Math. {\bf 342}, Springer-Verlag, Berlin, 1973, 474-486

\bibitem{Fe} {\sc I.\,B.\,Fesenko}, 
\textit{Sequential topologies and quotients of 
the Milnor $K$-groups of higher local fields}, 
{Algebra i Analiz} {\bf 13} (2001), no.3, 
198-221; English translation in: St. Petersburg Math. J. 
{\bf 13} (2002), issue 3, 485-501 

\bibitem{FV} {\sc I.\,Fesenko, S.\,Vostokov}, \textit{Local Fields and their Extensions}, 
Translations of Mathematical Monographs, vol.121, Amer. Math. Soc., Providence, 
Rhode Island, 2002

\bibitem{Fo} {\sc J.-M.Fontaine}, \textit{Representations $p$-adiques des corps locaux 
(1-ere partie)}. In: The Grothendieck Festschrift, A Collection of Articles 
in Honor of the 60th Birthday of Alexander Grothendieck, vol.\, II, 1990, 249-309 

\bibitem{Fu} {\sc T.\,Fukaya}, \textit{The theory of Coleman power series for $K_2$}, 
J.\,Algebraic Geom. {\bf 12} (2003) No.\,1, 1-80


\bibitem{MZh} {\sc A.\,I.\,Madunts, I.\,B.\,Zhukov}, 
\textit{Multidimensional complete fields: topology and other basic 
constructions}, Trudy S.Peterb. Mat. Obshch. (1995);
English translation in: {Amer. Math. Soc. Transl.}, (Ser.2) {\bf165}, 1-34



\bibitem{Ka} {\sc K.\,Kato} 
\textit{The explicit reciprocity law and the cohomology of Fontaine-Messing}, 
{Bull. Soc. Math. France} {\bf109} (1991) no.4, 397-441

\bibitem{La} {\sc F.\, Laubie} \textit{Extensions de Lie et groupes d'automorphismes 
de corps locaux}, Comp. Math. {\bf 67} (1988), 165-189

\bibitem{Ne} {\sc Ju.\,Neukirch} \textit{Class Field theory}, Springer-Verlag, 
Berlin and New York, 1986

\bibitem{Pa1} {\sc A.\,N.\,Parshin}, 
\textit{Class fields and algebraic $K$-theory. (Russian)}, 
{Uspekhi Mat. Nauk} {\bf30} (1975), 253-254; 

\bibitem{Pa2} {\sc A.\,N.\,Parshin}, \textit{Local class field theory. (Russian)}, 
Algebraic Geometry and its applications, Trudy Mat. Inst. Steklov {\bf165},(1985),  
143-170; English translation in: Proc. Steklov Inst. Math., 1985, issue 3, 157-185.

\bibitem{Pa3} {\sc A.\,N.\,Parshin}, \textit{Galois cohomology and Brauer group of 
local fields}, Trudy Mat. Inst. Steklov (1990); English translation 
in: Proc. Steklov Inst. Math., 1991, issue 4, 191-201.


\bibitem{Sch} {\sc T.\,Scholl}, \textit{Higher fields of norms and 
$(\varphi ,\Gamma )$-modules}, {Doc. Math.}, {Extra vol.}, (2006), 685-709 (electronic)


\bibitem{Sh} {\sc I.\,R.\,Shafarevich} \textit{A general reciprocity law}(Russian), 
Mat. Sb. {\bf 26(68)} (1950), 113-146


\bibitem{Vo1} {\sc S.\,Vostokov}, \textit{An explicit form of 
the reciprocity law}. {Izv. Akad. Nauk SSSR}, Ser. Mat. (1978); English translation in: 
Math. USSR Izv. {\bf 13} (1979), 557-588

\bibitem{Vo2} {\sc S.\,Vostokov}, {Explicit construction of the 
theory of class fields of a multidimensional local field}, 
{Izv. Akad. Nauk SSSR Ser. Mat.}, {\bf49}, (1985), 283-308

\bibitem{Ze}
{\sc S.Zerbes} \textit{The higher Hilbert pairing via (phi,G)-modules}. arXiv:0705.4269 

\bibitem{Zh1} {\sc I.\,Zhukov}, 
\textit{Higher dimensional local fields} 
(M\"unster, 1999), Geom. Topol. Monogr. (2000), no.3, 5-18








\end{thebibliography}
\end{document}